\documentclass[12pt,reqno]{amsart}
\usepackage{fullpage}

\newtheorem{theorem}{Theorem}[section]
\newtheorem{lemma}[theorem]{Lemma}

\newtheorem{cor}[theorem]{Corollary}

\usepackage{graphicx}
\usepackage{color}
\usepackage{subfigure}
\usepackage{amssymb}
\usepackage{amsmath}
\usepackage{colonequals}
\usepackage{hyperref}
\theoremstyle{definition}
\newtheorem{definition}[theorem]{Definition}

\newtheorem{example}[theorem]{Example}

\newtheorem{remark}[theorem]{Remark}

\renewcommand{\subset}{\subseteq}

\renewcommand{\epsilon}{\varepsilon}

\newcommand{\abs}[1]{\left|#1\right|}                   % Absolute value notation
\newcommand{\absf}[1]{|#1|}                             % small absolute value signs
\newcommand{\vnorm}[1]{\left\|#1\right\|}    % norm notation
\newcommand{\vnormf}[1]{\|#1\|}                         % norm notation, forced to be small
\newcommand{\mnorm}[1]{\left|#1\right|}    % matrix "norm"; absolute value
\newcommand{\mnormsq}[1]{\left|#1\right|^{2}}    % matrix "norm"; absolute value, squared
\newcommand{\Tr}[1]{\mbox{Tr}#1}

\newcommand{\N}{\mathbb{N}}
\newcommand{\E}{\mathbb{E}}

\newcommand{\R}{\mathbb{R}}
\newcommand{\C}{\mathbb{C}}
\renewcommand{\P}{\mathbb{P}}

\newcommand{\italicize}[1]{\textit {#1}}                % formatting commands for bibliography
\newcommand{\embolden}[1]{\textbf {#1}}
\newcommand{\cxhyper}{\{-1,1\}}   %symbol denoting the complex hypercube \{1,-1,\sqrt{-1},-\sqrt{-1}\}

\renewcommand{\Tr}{\mbox{\rm Tr}}
\newcommand{\Id}{I}

\newcommand{\ignore}[1]{}

\newcommand{\ip}{{\iota_p}}
\newcommand{\ii}{\iota}
\newcommand{\Inf}{\mathrm{Inf}}

\newcommand{\ncp}[1]{M_{#1}(\C)[X_1,\ldots,X_{m}]}

\newcommand{\hyconstv}{(2K-1)^{dK}c_{K}^{d}}  %put hypercontractive constant into a macro
\newcommand{\hyconst}{(2K-1)^{2dK}c_{K}^{d}}  %put hypercontractive constant into a macro
\newcommand{\hyconstmixed}{(2K-1)^{d}c_{K}^{d/(2K)}}  %put hypercontractive constant into a macro
\newcommand{\hyconstmixedk}{(2K-1)^{k}c_{K}^{k/(2K)}}  %put hypercontractive constant into a macro
\newcommand{\hyconsta}{(3^{4}c_{2})^{d}}  %case K=2 of above estimate
\newcommand{\hyconstb}{(5^{6}c_{3})^{d}}  %case K=3 of above estimate
\newcommand{\hyconstab}{(5^{6}c_{2}c_{3})^{d}}  %put hypercontractive constant into a macro

\newcommand{\Eb}{\mathop{\mathbb{E}}_{b_j\sim \mathcal{B}}}
\newcommand{\Ebi}{\mathop{\mathbb{E}}_{b_i\sim \mathcal{B}}}

\newcommand{\Eg}{\mathop{\mathbb{E}}_{G_j\sim \mathcal{G}}}
\newcommand{\Eu}{\mathop{\mathbb{E}}_{G_j\sim \mathcal{V}}}
\newcommand{\Egh}{\mathop{\mathbb{E}}_{\substack{G_k\sim \mathcal{G}\\ H_k\sim \mathcal{H}_p}}}

\newcommand{\Eivh}{\mathop{\mathbb{E}}_{\substack{G_k\sim \mathcal{V}\otimes\Id\\ H_k\sim \mathcal{H}_p}}}

\newcommand{\Eghi}{\mathop{\mathbb{E}}_{\substack{G_k\sim \mathcal{G}\otimes I\\ H_k\sim \mathcal{H}_p}}}

\newcommand{\snote}[1]{}
\newcommand{\rd}[1]{{#1}}

\begin{document}

\title{A Moment Majorization Principle for\\ Random Matrix Ensembles}

\author{Steven Heilman}
\address{Department of Mathematics, UCLA, Los Angeles, CA 90095-1555}
\email{stevenmheilman@gmail.com}

%68Q17- Computer science; Computational difficulty of problems (lower bounds, completeness, difficulty of approximation, etc.)
%60F17 - 	Probability; Functional limit theorems; invariance principles
%60E15  - probability; 	Inequalities; stochastic orderings
%47A50- operator theory; Equations and inequalities involving linear operators, with vector unknowns

\begin{abstract}
We prove a moment majorization principle for matrix-valued functions with domain $\{-1,1\}^{m}$, $m\in\mathbb{N}$. The principle is an inequality between higher-order moments of a non-commutative multilinear polynomial with different random matrix ensemble inputs, where each variable has small influence and the variables are instantiated independently.

This technical result can be interpreted as a noncommutative generalization of one of the two inequalities of the seminal invariance principle of Mossel, O'Donnell and Oleszkiewicz.  Applications to noncommutative noise stability and noncommutative anticoncentration are given.
\end{abstract}
\maketitle

\section{Introduction}

\subsection{A noncommutative moment majorization theorem}

We study matrix-valued functions $f$ with domain $\{-1,1\}^{m}$ within the context of probability theory and Fourier analysis.  More specifically, we study functions $f$ such that, for every $\sigma\in\{-1,1\}^{m}$, the operator norm of $f(\sigma)$ is at most $1$.  The special case when $f$ is valued in the two-point space $\{-1,1\}$ has been studied extensively within theoretical computer science \cite{kahn88}, but also in diverse areas such as combinatorics, isoperimetry \cite{talagrand94}, or social choice theory \cite{kalai02,mossel10,mossel12}.  (For a more comprehensive list of references and discussion, see e.g. the survey \cite{odonnell14b}.)
In applications to theoretical computer science, a function $f\colon\{-1,1\}^{m}\to\{-1,1\}$ can be used to represent an instance of a combinatorial optimization problem.  That is, the function $f$ can be thought of as a list of elements of $\{-1,1\}$, seen as a Boolean assignment to the $2^m$ variables of some constraint satisfaction problem. Functions with domain $\{-1,1\}^{m}$ and range the simplex $\{(x_{1},\ldots,x_{n})\in\R^{n}\colon \sum_{i=1}^{n}x_{i}=1,\,x_{1}\geq0,\ldots,x_{n}\geq0\}$ have also been considered \cite{khot09,khot11,isaksson11}.  Projecting $f$ onto each coordinate gives a family of functions with range $[0,1]$, so that similar tools to the Boolean case can be applied.
%Here our main application, and motivation, is to the noncommutative Grothendieck inequality (NCGI), an inequality which involves two orthogonal (in the real case) or unitary (in the complex case) matrix variables of fixed dimension. This setting leads us to consider \emph{matrix-valued functions} with domain $\{-1,1\}^{m}$ (considering $m>1$ will allow us to ``combine'' multiple instances of NCGI acting on partially overlapping sets of variables).

In many of the applications listed above a standard manipulation is to extend $f$ to a multilinear polynomial, so that the distribution of $f$ can be studied under different distributions on its domain, such as the standard Gaussian distribution. In our setting it is  natural (and, as we will see, for our purposes necessary) to investigate the behavior of  matrix-valued functions under distributions on their domain that allow the possibility for matrix variables. For any set $S$, let $M_n(S)$ denotes the $n\times n$ matrices with entries in $S$. Any $f:\{-1,1\}^m \to M_n(\C)$ can be extended to a multilinear polynomial in $m$ \emph{noncommutative} variables with matrix coefficients.  Consider for instance the case $m=2$ and the polynomial $f(\sigma_{1},\sigma_{2})=\sigma_{1}\sigma_{2}$, where $\sigma_{1},\sigma_{2}\in\{-1,1\}$.  Since the variables $\sigma_{1},\sigma_{2}$ commute, it is not necessary to specify the order in which the product of the variables is taken in $f$.  However, once $f$ is extended to  matrix variables $X_{1},X_{2}$, an ordering needs to be specified.  We adopt the convention of ordering matrix variables by increasing order, e.g. $f(X_{1},X_{2})=X_{1}X_{2}$.

Let $d,m,n$ be positive integers. For us, a noncommutative multilinear polynomial of degree $d$ in $m$ variables can be expressed as
$$Q(X_{1},\ldots,X_{m})\,=\,\sum_{S\subset\{1,\ldots,m\}\colon\abs{S}\leq d}\widehat{Q}(S)\,\prod_{i\in S}X_{i},$$
where $\widehat{Q}(S)$ is an $n\times n$ complex matrix for every $S\subset\{1,\ldots,m\}$, $X_{1},\ldots,X_{m}$ are noncommutative $n\times n$ matrix variables, and the product $\prod_{i\in S}X_{i}$ is always taken in increasing order.  For example, $\prod_{i\in\{1,2\}}X_{i}=X_{1}X_{2}$. Noncommutative polynomials appear in many other contexts, most notably, within free probability \cite[Theorem 3.3]{voiculescu91}. In addition there is a general theory of so-called nc-functions \cite{kal14}, but this theory does not seem to apply to the noncommutative polynomials we consider here.  (An nc function $h$ is a function defined on matrices of any dimension, such that, for any $n\geq1$, and for any $n\times n$ matrices $A,B,C$ such that $C$ is invertible, $h(CAC^{-1})=Ch(A)C^{-1}$ and $h(A\oplus B)=h(A)\oplus h(B)$. Neither property is satisfied by a general matrix-valued non-commutative polynomial as defined below.)

Our main goal consists in bounding the moments of $Q$ for different random matrix  distributions in the domain, when all partial derivatives of $Q$ are small (i.e. when $Q$ has small influences).  In particular, we would like to say that the moments of polynomials $Q$ with small influences under Gaussian random matrix inputs are close to the moments of $Q$ under uniform $\{-1,1\}^{m}$ inputs.
Unfortunately, this task is in general impossible.  For example, consider the linear polynomial $Q(X_{1},\ldots,X_{m})=(X_{1}+\cdots+X_{m})/\sqrt{m}$.  For any square matrix $A$, let $\abs{A}=(AA^{*})^{1/2}$.  Let $b_{1},\ldots,b_{m}$ be i.i.d. uniform random variables in $\{-1,1\}$, and let $I$ denote the $n\times n$ identity matrix.  Then $\E(1/n)\Tr\abs{Q(b_{1}I,\ldots,b_{m}I)}^{4}=3-2/m$. On the other hand, let $G_{1},\ldots,G_{m}$ be $n\times n$ independent Wigner matrices with real Gaussian entries. In this case $Q(G_{1},\ldots,G_{m})$ is equal in distribution to $G_{1}$, and in particular $\lim_{n\to\infty }(1/n)\Tr\abs{Q(G_{1},\ldots,G_{m})}^{4}=2$, by the semicircle law. Thus even though the first and second moments of the input distributions match, i.e. $\E b_{1}=0$, $\E b_{1}^{2}=1$, $\E G_{1}=0$ and $\E G_{1}G_{1}^{*}=I$, the associated moments of $Q$ can be very different.

In summary, a general invariance principle cannot hold in this noncommutative setting.  We could instead try to prove a weaker statement such as: the moments of $Q$ under noncommutative inputs with $\E G_{1}=0$ and $\E G_{1}G_{1}^{*}=I$ are bounded by the moments of $Q$ under Boolean inputs. We call such a statement a \emph{moment majorization theorem}. Unfortunately, this is also not true in general, as we now show.  For any $1\leq i\leq n$, let $A_{i}$ be the $n\times n$ matrix with a $1$ in the first row and $i^{th}$ column, and zeros in all other entries.  Then for any $1\leq j,k\leq n$ with $j\neq k$, we have $A_{j}A_{k}^{*}=0$.  Consider the linear polynomial
$$Q(X_{1},\ldots,X_{n})=\sum_{i=0}^{n-1}B^{i}AX_{i},$$
which is such that
$
\E \Tr\abs{Q(b_{1},\ldots,b_{n})}^{4} = n.
$
Now let $H_{1}$ be a uniformly random Haar distributed $n\times n$ unitary matrix, and let $H_{2},\ldots,H_{n}$ be independent copies of $H_{1}$.  Then
$$
\E \Tr|Q(H_1,\ldots,H_{n})|^4
=n\Tr\abs{A_{1}}^{4}+n(n-1)\E\Tr(\abs{A_{1}}^{2}H_{1}\abs{A_{1}}^{2}H_{1}^{*})
=n+n(n-1)/n=2n-1.
$$

Since a general noncommutative majorization principle cannot hold we instead establish a limited moment majorization theorem, which will nevertheless be sufficient for our applications. We make two changes.
We first increase the dimension of the random matrix inputs $G_{1},\ldots,G_{m}$.  That is, we  allow the variables of $Q$ to take values in the set of $p\times p$ matrices with $p>n$ by defining the $p\times p$ matrix
$$\iota(A)\,=\,\begin{pmatrix}A & 0\\ 0 & 0\end{pmatrix}$$
for any $n\times n$ matrix $A$, and
$$Q^{\iota}(X_{1},\ldots,X_{m})=\sum_{S\subset\{1,\ldots,m\}} \iota\big(\widehat{Q}(S) \big)\,\prod_{i\in S}X_{i},$$
where $X_{1},\ldots,X_{m}$ are noncommutative $p\times p$ matrix variables. Second, we randomly rotate $G_{1},\ldots,G_{m}$ by $p\times p$ Haar-distributed random unitary matrices $H_{1},\ldots,H_{m}$.

We state one particular variant of our noncommutative moment majorization theorem.  \rd{Let $n$ be a positive integer and let $p$ be a positive multiple of $n$.}  We write $H\sim \mathcal{H}$ to denote a $p\times p$ Haar-distributed random unitary matrix, $b\sim\mathcal{B}$ for a uniformly random $b\in\{-1,1\}$, and  $G\sim \mathcal{G}$ for any random variable taking values in $M_n(\C)$ such that $\E G=0$ and $\E GG^{*}=\Id$. We also write $G'\sim \mathcal{G}\otimes \Id$ to denote $G'=\rd{G\otimes \Id}$ with $G\sim\mathcal{G}$, \rd{so that $G'$ is a $p\times p$ matrix}. We use the succinct notation $G_{i}\sim\mathcal{G}$ to denote a collection $G_{1},\ldots,G_{m}$ of independent random matrices with distribution $\mathcal{G}$, and denote $Q(G_{1},\ldots,G_{m})$ as $Q\{G_{i}\}$. The operator norm of a matrix $A$ is denoted $\vnorm{A}$.

\begin{theorem}[\embolden{Noncommutative Fourth Moment Majorization}]\label{firstthm}
Let $Q$ be a noncommutative multilinear polynomial of degree $d$ in $m$ variables such that $\vnorm{Q(\sigma)}\leq1$ for all $\sigma\in\{-1,1\}^{m}$.  Let $\tau\colonequals\max_{i=1,\ldots,m}\sum_{S\subset\{1,\ldots,m\}\colon i\in S}\Tr(\widehat{Q}(S)\widehat{Q}(S)^{*})$ be the maximum influence of $Q$.  Let $G_{i}\sim\mathcal{G}$ and $c_{2}\geq 1$ such that $\vnorm{\E (G_{1}G_{1}^{*})^{2}}\leq c_{2}$.  Then
\begin{equation}\label{zero0}
\Eghi\frac{1}{n}\Tr\abs{Q^{\iota}\{G_{i}H_{i}\}}^{4}
\leq\Ebi\frac{1}{n}\Tr\mnorm{Q\{b_{i}\}}^{4}
+8(8c_{2})^{4d}n^{4}\tau^{1/4}+O_{m,n}(p^{-1/2}).
\end{equation}
\end{theorem}
\begin{remark}
This Theorem is a special case of Theorem \ref{invlemAb} below.  The term $8(8c_{2})^{4d}n^{4}\tau^{1/4}$ on the right-hand side of \eqref{zero0} can be replaced by $(8c_{2})^{4d}\tau$ by additionally assuming that $\E a_{1}a_{2}a_{3}=0$ for any $a_{1},a_{2},a_{3}$ which are (possibly repeated) entries of $G_{1}$.  We omit the proof of this strengthened statement, since the details are essentially identical to the proof of Theorem \ref{thm:invariance-convex}.
\end{remark}

\begin{remark}
Note that, although $Q^{\iota}$ takes values in the set of $p\times p$ matrices, the trace is normalized by $1/n$, and $Q^{\iota}$ still ``acts like'' an $n\times n$ matrix. In particular the moments of $Q^{\iota}$ do not become arbitrarily small in general; for instance it holds   (see Lemma \ref{lemma80} below) that
$$\Eghi\frac{1}{n}\Tr\abs{Q^{\iota}\{G_{i}H_{i}\}}^{2}\,=\,\Ebi\frac{1}{n}\Tr\mnorm{Q\{b_{i}\}}^{2}.$$
\end{remark}

Theorem~\ref{firstthm} shows that the fourth moment of $Q$ with appropriate random matrix inputs is bounded by the fourth moment of $Q$ with Boolean inputs.  We provide a variant of this majorization principle for higher order moments (Theorem \ref{invlemAb}) and for increasing test functions (Theorem \ref{thm:invariance-convex}).  These majorization principles all have dependence on the degree of the polynomial, but we also give degree-independent bounds for polynomials whose higher order coefficients decay at an exponential rate (Corollary \ref{cor:smooth-invariance}).

Although majorization principles such as Theorem \ref{firstthm} involve the trace norm of a polynomial, we can obtain bounds on the operator norm of $Q$ in the following way.  Let $t\in\R$ and consider the function $t\mapsto(\max(0,\abs{t}-1))^{2}$.  This function can be applied to self-adjoint matrices via spectral calculus, and $A\mapsto(\max(0,(AA^{*})^{1/2}-1))^2=0$ if the singular values of $A$ are all bounded by $1$.  In particular, if $\vnorm{Q(\sigma)}\leq1$ for all $\sigma\in\{-1,1\}^{m}$ then $(\max(0,(Q(\sigma)Q(\sigma)^{*})^{1/2}-1))^2=0$ for all $\sigma\in\{-1,1\}^{m}$.  The following majorization principle gives control on the operator norm of $Q$ when we substitute appropriate random matrices into the domain of $Q$.

\begin{theorem}[\embolden{Noncommutative Operator Norm Majorization}]\label{secondthm}
Let $Q$ be a noncommutative multilinear polynomial of degree $d$.  Suppose $\vnorm{Q(\sigma)}\leq1$ for all $\sigma\in\{-1,1\}^{m}$.  Let $\tau\colonequals\max_{i=1,\ldots,m}\sum_{S\subset\{1,\ldots,m\}\colon i\in S}\Tr(\widehat{Q}(S)\widehat{Q}(S)^{*})$ be the maximum influence of $Q$.  Let $G_{i}\sim\mathcal{G}$ and $c_{2},c_{3}\geq 1$ such that $\vnorm{\E (G_{1}G_{1}^{*})^{2}}\leq c_2$ and $\vnorm{\E (G_{1}G_{1}^{*})^{3}}\leq c_3$.  Then
\begin{equation}
\Eghi\frac{1}{n}\Tr\big(\max(0,\abs{Q^{\iota}\{G_{i}H_{i}\}}-1)\big)^{2}
\leq(8c_{2}c_{3})^{9d}n^{1/2}\tau^{1/6}+O_{m,n}(\tau^{-1/3}p^{-1/2}).
\end{equation}
\end{theorem}

This Theorem appears below in Theorem \ref{thm:invariance}.  We will use this theorem to argue that under the proper normalization condition a low-influence polynomial typically maps random Gaussian matrix inputs to matrices of norm not much larger than $1$. The ensemble that will be of most interest for us is the following.

\begin{example}\label{ex:v}
Let $N>0$.  Let $V_{1},\ldots,V_{N}$ be $n\times n$ complex matrices such that $\sum_{i=1}^{N}V_{i}V_{i}^{*}=1$.  Let $g_{1},\ldots,g_{N}$ be i.i.d. standard complex Gaussian random variables (so that $\E g_{1}=0$ and $\E\abs{g_{1}}^{2}=1$.)  Define $G_{1}=\sum_{i=1}^{N}g_{i}V_{i}$, and let $G_{2},\ldots,G_{m}$ be independent copies of $G_{1}$.  Then it follows from~\cite[Corollary 2.8]{haagerup99} that the random matrices $G_{1},\ldots,G_{m}$ satisfy the hypothesis of Theorem \ref{firstthm} with $c_2=2$.  In the case that $V_{1},\ldots,V_{N}$ are real matrices and $g_{1},\ldots,g_{N}$ are i.i.d. standard real Gaussians, the hypothesis of Theorem \ref{firstthm} is satisfied with $c_2=4\cdot2=8$, as follows from the complex case and the inequality $\vnorm{\E(\Re(G_{1})\Re(G_{1})^{*})^{2}}\leq \vnorm{\E(G_{1}G_{1}^{*})^{2}}\leq2$ when $g_{1},\ldots,g_{N}$ are complex.
\end{example}

Here are some other examples of random matrix ensembles satisfying the hypothesis of Theorem \ref{secondthm}.
\begin{example}
Example~\ref{ex:v} specifically applies to Gaussian Wigner matrices as follows.  Let $U_{1},\ldots,U_{n^{2}}$ be $n\times n$ matrices such that these matrices are the standard orthonormal basis of $\C^{n^{2}}$.  Then $G_{1}$ is a Wigner matrix and the hypothesis of Theorem \ref{firstthm} is satisfied with $c_2=2$.  Similarly, let $U_{1},\ldots,U_{n(n+1)/2}$ be $n\times n$ matrices such that these matrices are the standard orthonormal basis of symmetric $n\times n$ matrices.  Then $G_{1}$ is a Wigner matrix and the hypothesis of Theorem \ref{firstthm} is satisfied with $c_2=2$ (see \cite[Theorem II.11]{davidson01} or \cite[Theorem 5.32]{ver12}).
\end{example}

\begin{example}
Let $G_{1},\ldots,G_{m}$ be i.i.d. $n\times n$ Haar-distributed random unitary matrices.  Then the random matrices $G_{1},\ldots,G_{m}$ satisfy the hypothesis of Theorem \ref{firstthm} with $c_2=1$.
\end{example}

We give a brief overview of the strategy of the proof of Theorem \ref{secondthm} and its generalization, Theorem \ref{invlemAb}. In order to prove the majorization principle, we first establish some basic facts about Fourier analysis of matrix-valued functions $f$ in Section \ref{sec:notation}.  In particular, starting from a function $f\colon\{-1,1\}^{m}\to M_{n}(\C)$, we extend $f$ to a noncommutative multilinear polynomial $Q=Q_{f}$ of $m$ variables.

For these polynomials $Q$, we consider a few different inner products, norms, derivatives, Plancherel identities, and we also define the Ornstein-Uhlenbeck semigroup.  We then prove a noncommutative hypercontractive inequality for such polynomials $Q$.  This hypercontractive inequality is developed in Section \ref{sec:invariance} and proven in Theorem \ref{k2hyper}.  It can also be considered a polynomial generalization of the matrix Khintchine inequality.

To prove the noncommutative majorization principle, we use the Lindeberg replacement method from \cite{mossel10,chat06}, but in our (matrix-valued) polynomials, we are replacing one random variable with one random matrix, one at a time.  The ``second order'' terms in this replacement have a noncommutative nature which introduces an error.  Instead of trying to bound this error (which seems difficult or impossible in general), we choose particular random matrices such that the ``second order'' terms are small.  This is accomplished by replacing an $n\times n$ random matrix $X$ with a much larger $p\times p$ matrix $\begin{pmatrix} X & 0\\ 0 & 0\end{pmatrix}H$, where $H$ is a uniformly random $p\times p$ unitary matrix and $p>n$.  When $p\to\infty$, the noncommutative ``second order'' errors in the Lindeberg replacement vanish.

\subsection{Applications}

The commutative invariance principle \cite{rotar79,chat06,mossel10} implies that if $Q$ is a commutative multilinear polynomial with small derivatives (i.e. small influences), then the distribution of $Q$ on i.i.d. uniform inputs in $\{-1,1\}$ is close to the distribution of $Q$ on i.i.d. standard Gaussian random variables.  A more general statement can be made for more general functions and distributions; for details see e.g. \cite{mossel10}.  An invariance principle can also be considered as a concentration inequality, generalizing the central limit theorem with error bounds (i.e. the Berry-Ess\'{e}en Theorem).  For other variants of invariance principles, see \cite{mossel10b,isaksson11}.

The form of the invariance principle given in \cite{mossel10} is proven by a combination of the Lindeberg replacement argument and the hypercontractive inequality \cite{bonami70,nelson73,gross75}.  That is, one replaces one argument of $Q$ at a time, adding up the resulting errors and controlling them via the hypercontractive inequality. One version of hypercontractivity says that a higher $L_{q}$ norm of a polynomial is bounded by a lower $L_{p}$ norm of that polynomial, where $q>p$, with a bound dependent on the degree of the polynomial $Q$.  For example, if $Q$ has degree $d$, then the $L_{4}$ norm of $Q$ is bounded by $9^{d}$ times the $L_{2}$ norm of $Q$.

The commutative invariance principle has seen many applications \cite{odonnell14b,odonnell14} in recent years.  Here is a small sample of such applications and references: isoperimetric problems in Gaussian space and in the hypercube \cite{mossel10,isaksson11}, social choice theory, Unique Games hardness results \cite{khot07,isaksson11}, analysis of algorithms \cite{rag15}, random matrix theory \cite{mendelson14}, free probability \cite{nourdin10}, optimization of noise sensitivity \cite{kane14}. The Lindeberg replacement argument itself has many applications, e.g. in proving the universality of eigenvalue statistics for Wigner matrices \cite[Theorem 15]{tao11}.

We anticipate that our noncommutative majorization principle will find similar applications.  Even though it is impossible to prove a noncommutative invariance principle in general, most applications of the commutative invariance principle only involve one direction of the inequality.  That is, most applications of the invariance principle are really just applications of a majorization principle such as Theorem \ref{firstthm} or \ref{secondthm}.

To demonstrate further applications of our majorization principle, we show in Section \ref{secnoise} that one of the two main parts of the proof of the Majority is Stablest Theorem from \cite{mossel10} can be extended to the noncommutative setting.  Then, in Section \ref{secanti}, we demonstrate a (probably sub-optimal) anti-concentration estimate for noncommutative multilinear polynomials.  Both of these results proceed as in \cite{mossel10} by replacing their invariance principle with our majorization principle.

Since majorization principles such as Theorems \ref{firstthm} and \ref{secondthm} show the closeness of one distribution to another, these statements could be fit into the ``concentration of measure'' paradigm.  The paper \cite{meckes12} proves a concentration inequality for noncommutative polynomials, but these methods seem insufficient to prove a majorization principle.  An invariance principle has been proven in the free probability setting~\cite{deya14}, but the details of exactly what polynomials can be dealt with, and which distributions can be handled, seem incomparable to our majorization principle.

\begin{remark}\label{negsub}
We remark that, although it may be tempting to try to prove Theorem \ref{secondthm} from the commutative invariance principle of \cite{mossel10}, there seems to be no straightforward way to accomplish this task.  For example, we could interpret each entry in the output of a noncommutative polynomial $Q$ as a commutative polynomial function of the inputs.  But then in order to control $\Tr\abs{Q}^{4}$, we would need information on the joint distribution of the entries of $Q$, which is not provided by the invariance principle of \cite{mossel10}.
\end{remark}

%============================%
\section{Definitions, background and notation}
\label{sec:notation}
%============================%

\subsection{Matrices}

For $n\in \N$ we denote the set of $n$ by $n$ matrices by $M_{n}(\C)$.  We use $M_{n\times n}(\C)$ in place of $M_{n^2}$ to denote $n^2$ by $n^2$ matrices when we wish to emphasize that a specific tensor decomposition of the space $\C^{n^2} = \C^n \otimes \C^n$ on which the matrix acts has been fixed.

For $A\in M_n(\C)$, $\|A\|$ is the operator norm of $A$ (the largest singular value). We denote by $A^*$ the conjugate-transpose. The absolute value is $\abs{A}=(AA^{*})^{1/2}$. We use $I$ to denote the identity matrix.

Any real function $f:\R\to\R$ can be applied to a Hermitian matrix $A$ by diagonalizing $A$ and applying $f$ to the eigenvalues of $A$. Define $\mathrm{Chop}:\R\to\R$ as $\mathrm{Chop}(t) = t$ if $|t|\leq 1$, $\mathrm{Chop}(t)=1$ if $t\geq 1$, and $\mathrm{Chop}(t)=-1$ if $t\leq -1$.

\subsection{Asymptotic Notation}

Let $a,b,c\in\R$.  We write $a=O_{c}(b)$ if there exists a constant $C(c)>0$ such that $\abs{a}\leq C(c)\abs{b}$.  We write $a\leq O_{c}(b)$ if there exists a constant $C(c)>0$ such that $a\leq C(c)b$.

\subsection{Random variables and expectations}
\label{sec:random-variables}

The following notational conventions will be useful when working with functions of multiple variables. Let $m,n\in\N$ and let $S,T$ be arbitrary sets.  Let $f:S^m\to T$ and let $X_1,\ldots,X_n$ be independent random variables with the same distribution $\mathcal{X}$ taking values in $S$. Then we will denote $\E f(X_1,\ldots,X_m)$ by $\E_{X_j\sim\mathcal{X}} f\{X_j\}$; more generally the curly bracket notation $f\{A_j\}$ will be used to denote $f(A_1,\ldots,A_m)$.

We will use the following ensembles. Let $n,N\in\N$, \rd{and let $p$ be a positive multiple of $n$}.  $H\sim \mathcal{H}$ denotes a $p\times p$ Haar-distributed random unitary matrix, where the dimension $p$ will always be clear from context. $b\sim\mathcal{B}$ denotes a uniformly random $b\in\{\pm 1\}$.  $G\sim \mathcal{G}$ denotes any random variable taking values in $M_n(\C)$ such that
 $\E G=0$ and $\E GG^{*}=\Id$.  And $G'\sim\mathcal{G}\otimes \Id$ denotes $G'=\rd{G\otimes\Id_{p/n}\in M_{p}(\C)}$, where $G\sim\mathcal{G}$.  $G\sim\mathcal{V}$ denotes a random variable distributed as $G = \sum_{i=1}^N g_i V_i$, where $g_{1},\ldots,g_{N}$ are i.i.d. standard complex Gaussian random variables and $V_1,\ldots,V_N$ are $n\times n$ complex matrices satisfying $\sum_{i=1}^N V_iV_i^* = \sum_{i=1}^N V_i^*V_i=\Id_n$.  And $G'\sim\mathcal{V}\otimes\Id$ denotes $G'=\rd{G\otimes \Id_{p/n}\in M_{p}(\C)}$, where $G\sim\mathcal{V}$.  Whenever $\mathcal{V}$ is used the matrices $V_{1},\ldots,V_{N}$ will be clear from context.  We also sometimes write $G_j\sim\mathcal{D}$ to mean $G_1,\ldots,G_m$ are independent random variables with distribution $\mathcal{D}$, where again $m$ will always be clear from context.

\subsection{Fourier expansions}

Let $n,m\in\N$ and $f,h\colon\cxhyper^{m}\to M_{n}(\C)$.  We consider the inner product
$$\langle f,h\rangle\colonequals \Eb \Tr\big(f\{b_j\}(h\{b_j\})^{*}\big) \,=\,2^{-m}\sum_{\sigma\in \cxhyper^{m}}\mathrm{Tr}(f(\sigma)h(\sigma)^{*}).$$
  Given $S\subset\{1,\ldots,m\}$ and $\sigma=(\sigma_{1},\ldots\sigma_{m})\in\cxhyper^{m}$, define $$W_{S}(\sigma)\colonequals\prod_{i\in S}\sigma_{i}.$$
The set of functions $\{W_{S}\}_{S\subset\{1,\ldots,n\}}$ forms an orthonormal basis for the space of functions from $\cxhyper^{m}$ to $M_{n}(\C)$, when it is viewed as a vector space over $\C$ with respect to the inner product $\langle\cdot,\cdot\rangle$.

Let $\widehat{f}(S)\colonequals2^{-m}\sum_{\sigma\in\cxhyper^{m}}f(\sigma)W_{S}(\sigma)$ be the Fourier coefficient of $f$ associated to $S$.  Note that $\widehat{f}(S)\in M_{n}(\C)$.  Then $f=\sum_{S\subset\{1,\ldots,m\}}\widehat{f}(S)W_{S}$, and
\begin{flalign}
2^{-m}\sum_{\sigma\in\cxhyper^{m}}f(\sigma)(g(\sigma))^{*}
&=2^{-m}\sum_{\sigma\in\cxhyper^{m}}\sum_{S,S'\subset\{1,\ldots,m\}}\widehat{f}(S)(\widehat{g}(S'))^{*}W_{S}(\sigma)W_{S'}(\sigma)\nonumber\\
%&=2^{-k}\sum_{\sigma\in\{1,-1,i,-i\}^{k}}\sum_{S,S'\in\{1,\ldots,m\}^{k}}\widehat{f}(S)\widehat{g}(S')W_{S\Delta S'}(\sigma)\\
&=\sum_{S,S'\subset\{1,\ldots,m\}}\widehat{f}(S)(\widehat{g}(S'))^{*}\cdot2^{-m}\sum_{\sigma\in\cxhyper^{m}}W_{S}(\sigma)W_{S'}(\sigma)\nonumber\\
&=\sum_{S\subset\{1,\ldots,m\}}\widehat{f}(S)(\widehat{g}(S))^{*}.\label{eqplan}
\end{flalign}

\subsection{Noncommutative polynomials}

Let $n,m\in \N$ be integers.
We consider noncommutative multilinear polynomials $Q\in \ncp{n}$, where $X_{1},\ldots,X_{m}$ are noncommutative indeterminates.
Monomials are always ordered by increasing order of the index, e.g. $X_1X_2$ and not $X_2X_1$. Any such polynomial can be expanded as
\begin{equation}\label{eq:def-ncpoly}
Q(X_1,\ldots,X_{m}) = \sum_{S\subset\{1,\ldots,m\}\colon\abs{S}\leq d} \widehat{Q}(S) \prod_{i\in S}X_{i},
\end{equation}
where $\widehat{Q}(S)\in M_n(\C)$ for all $S\subset\{1,\ldots,m\}$ and $0\leq d \leq m$ is the degree of $Q$, defined as $\max\{\abs{S}\colon\widehat{f}(S)\neq0,\,S\subset\{1,\ldots,m\}\}$.

Let $f\colon\cxhyper^{m}\to M_{n}(\C)$.
Define the (non-commutative) \italicize{multilinear polynomial} $Q_{f}\in \ncp{n}$  associated to $f$ by
\begin{equation}\label{eq:def-fpoly}
Q_{f}(X_{1},\ldots,X_{m})\colonequals\sum_{S\subset\{1,\ldots,m\}}\widehat{f}(S)\prod_{i\in S}X_{i}.
\end{equation}

\subsection{Partial Derivatives}
Let $f\colon\{-1,1\}^{m}\to M_{n}(\C)$.  Let $i\in\{1,\ldots,m\}$. Define the $i^{th}$ \italicize{partial derivative} of $f$ by
\begin{flalign}
&\partial_{i}f(\sigma)
\colonequals\sum_{S\subset\{1,\ldots,m\}\colon i\in S}\widehat{f}(S)W_{S}(\sigma)
=\frac{1}{2}(f(\sigma)-f(\sigma_{1},\ldots,-\sigma_{i},\ldots,\sigma_{m})),
\end{flalign}
and the $i^{th}$ \italicize{influence} of $f$ by
\begin{equation}\label{six0}
\mathrm{Inf}_{i}f\colonequals\sum_{S\subset\{1,\ldots,m\}\colon i\in S}\mathrm{Tr}(\widehat{f}(S)\widehat{f}(S)^{*}).
\end{equation}

Note that by~\eqref{eqplan}, $\mathrm{Inf}_{i}f=\langle\partial_{i}f,\partial_{i}f\rangle$ and $\sum_{i=1}^{m}\mathrm{Inf}_{i}f=\sum_{S\subset\{1,\ldots,m\}}\abs{S}\mathrm{Tr}(\widehat{f}(S)(\widehat{f}(S))^{*})$.

\subsection{Ornstein-Uhlenbeck semigroup}

For $f,h:\{-1,1\}^m\to M_n(\C)$ define their convolution $f*h$ by
\begin{equation}\label{one31}
f*h(\sigma)\colonequals 2^{-m}\sum_{\omega\in\{-1,1\}^{m}}f(\sigma\cdot \omega^{-1})h(\omega)
=\sum_{S\subset\{1,\ldots,m\}}\widehat{f}(S)\widehat{h}(S)W_{S}(\sigma),\quad\forall\,\sigma\in\{-1,1\}^{m}.
\end{equation}

Here $\sigma\cdot\omega$ denotes the componentwise product of $\sigma$ and $\omega$ and $\omega^{-1}$ denotes the multiplicative inverse of $\omega$, so that $\omega\cdot\omega^{-1}=(1,\ldots,1)$.

Let $0<\rho<1$. For any $\sigma=(\sigma_{1},\ldots,\sigma_{m})\in\{-1,1\}^{m}$, let
\begin{equation}\label{one30}
R_{\rho}(\sigma)\colonequals\prod_{j=1}^{m}(1+\rho\sigma_{j})=\sum_{S\subset\{1,\ldots,m\}}\rho^{\abs{S}}W_{S}(\sigma).
\end{equation}
Let $f\colon\{-1,1\}^m\to M_n(\C)$.  Define the Ornstein-Uhlenbeck semigroup $T_{\rho}f$ by
\begin{equation}\label{one32}
T_{\rho}f(\sigma)\colonequals\sum_{S\subset\{1,\ldots,m\}}\rho^{\abs{S}}\widehat{f}(S)W_{S}(\sigma)=f*R_{\rho}(\sigma),\quad\forall\,\sigma\in\{-1,1\}^{m}.
\end{equation}

\subsection{Truncation of Fourier Coefficients (or Littlewood-Paley Projections)}

Let $f\colon\{-1,1\}^{m}\to M_{n}(\C)$, and $Q_{f}\in\ncp{n}$ the multilinear polynomial associated to $f$.  Let $d\in\N$.  Let $P_{d}$ denote \italicize{projection onto the level-$d$ Fourier coefficients}.  That is, $\forall$ $\sigma\in\{-1,1\}^{m}$, $\forall$ $X_{1},\ldots,X_{m}\in M_{n}(\C)$,
$$P_{d}f(\sigma)\colonequals\sum_{S\subset\{1,\ldots,m\}\colon\abs{S}=d}\widehat{f}(S)W_{S}(\sigma),
\qquad P_{d}Q_{f}(X_{1},\ldots,X_{m})\colonequals\sum_{S\subset\{1,\ldots,m\}\colon\abs{S}=d}\widehat{f}(S)\prod_{i\in S}X_{i}.$$
Let $P_{\leq d} \colonequals \sum_{i\leq d} P_i$ denote projection onto the Fourier coefficients of degree at most $d$. Denote $P_{>d}f\colonequals f-P_{\leq d}f$, $P_{>d}Q_{f}\colonequals Q_{f}-P_{\leq d}Q_{f}$.

\subsection{Embeddings}

\begin{definition}[\embolden{Matrix embedding}]\label{def:embedding}
Let $A\in M_{n}(\C)$ and $p\geq n$. Define the embedding $\ip\colon M_{n}(\C)\to M_{p}(\C)$ by
$$\ip(A)\colonequals\begin{pmatrix}A & 0\\ 0 & 0\end{pmatrix} \in M_p(\C).$$
If $B = \sum_{k=1}^{N} C_k\otimes D_k \in M_{n\times n}(\C)$, extend this definition to
$$ \ip(B) \colonequals \sum_{k=1}^{N} \ip(C_k)\otimes \ip(D_k) \in M_{p\times p}(\C).$$
Lastly, if $f\colon\{-1,1\}^{m}\to M_{n}(\C)$ define $\ip(f):\cxhyper^{m}\to M_{p}(\C)$ by
$$\ip(f)(\sigma)=\ip(f(\sigma))\qquad\forall\,\sigma\in\cxhyper^{m}.$$
We will sometimes denote the same quantities by $A^\ii$, $B^\ii$ and $f^\ii$ respectively, leaving the dependence of $\iota$ on $p$ and $n$ implicit for clarity of notation.
\end{definition}

Note that if $Q=\sum_{S\subset\{1,\ldots,m\}} \widehat{Q}(S)\prod_{i\in S}X_{i} \in \ncp{n}$ is a noncommutative polynomial the last item in Definition~\ref{def:embedding} is equivalent to defining
$$Q^{\ii}  = \sum_{S\subseteq\{1,\ldots,m\}}\ip(\widehat{Q}(S)) \prod_{i\in S}X_{i} \in \ncp{p}.$$
If moreover $Q_f$ is defined from $f\colon\cxhyper^{m}\to M_{n}(\C)$ as in~\eqref{eq:def-fpoly} then
$$Q_f^{\ii}  = \sum_{S\subseteq\{1,\ldots,m\}}\ip(\widehat{f}(S)) \prod_{i\in S}X_{i} \in \ncp{p}.$$

\subsection{Coordinate projections}

\begin{definition}\label{def14}
Let $U\in M_{n}(\C^{N})$ with $UU^{*}=U^{*}U=I$ and $Q\in\ncp{n}$ a noncommutative polynomial. We denote the \italicize{Gaussian} $L_{2}$ \italicize{norm} of $Q$ associated to the ensemble $\mathcal{G}$ by
\begin{flalign*}
\vnorm{Q}_{L_{2},\mathcal{G}}\colonequals
\Big(\Egh\Tr \big|Q^{\ii}\{G_k H_k\}\big|^2\Big)^{1/2}.
\end{flalign*}
\end{definition}

Let $\mathcal{N}(0,1)$ denote the standard complex Gaussian distribution for a random variable.  For integers $m,N$ let $\E_{g_{ij}\sim\mathcal{N}(0,1)}$ denote expectation with respect to $g_{1,1},\ldots,g_{m,N}$ i.i.d. $\mathcal{N}(0,1)$.
\subsection{Bounds on random polynomials}
\label{sec:preliminaries}
%=======================%

The key difference between the random matrices $G_{i}\sim\mathcal{G}$ and $\rd{(G_{i}\otimes I)}H_{i}$ where $H_{i}\sim\mathcal{H}_{p}$ is that the matrices $\rd{(G_{i}\otimes I)}H_{i}$ behave well with respect to matrix products.  This property is exploited in Corollary \ref{cor:rs-ev-bound} below.

\begin{lemma}\label{lem:ev-bound}
Let \rd{$p$ be a positive multiple of $n$}, let $A,B\in M_{\rd{p}}(\C)$.  \rd{Let $I=I_{n}$ be the $n\times n$ identity matrix.}  Then
$$ \E_{H\sim\mathcal{H}_{p}}\big\| \rd{I^\ii} AHB\rd{I^\ii} \big\|^2 \,\leq\, \frac{n^{2}}{p}\|A\|^2\|B\|^2.$$
\end{lemma}

\begin{proof}
The nonzero eigenspace $K$ of $HB\rd{I^\ii} B^{*}H^*$ is a uniformly distributed subspace of dimension at most $n$ of $\C^p$. Given a unit vector $x$, the squared norm of the projection of $x$ on $K$ has expectation at most $n/p$.
% less than $5/\sqrt{p}$ with probability at least $1-e^{-\sqrt{p}}$ (see e.g.~\cite[Lemma 1]{laurent2000adaptive}).
Applying this to the eigenvectors of $A\rd{^{*}I^{\ii}A}$,
$$\E_{H\sim\mathcal{H}_{p}} \Tr(\rd{I^{\ii}}A H B \rd{I^{\ii}}B^* H^* A^*) \leq \frac{n}{p} \|B\|^2 \Tr(A^{*}\rd{I^{\iota}}A)\leq\frac{n^{2}}{p}\vnorm{B}^{2}\vnorm{A}^{2}. $$

%To conclude, use that $\|X \|^2 \leq \Tr(XX^*)$, for $X=A^{\iota}HB^{\iota}$.
\end{proof}

\begin{cor}\label{cor:rs-ev-bound}
Let $R,S\in\ncp{n}$ be multilinear polynomials not depending on the $j$-th variable such that $\Eb \frac{1}{n}\Tr|R\{b_j\}|^2 \leq 1$ and $\Eb \frac{1}{n}\Tr|S\{b_j\}|^2 \leq 1$. Then
$$\Eghi\big\|\left(S\rd{^{\ii}}\{G_i H_i \}\right)G_k H_k\left(R\rd{^{\ii}}\{G_iH_i\}\right)^*\big\| = O_{n,m}\big(p^{-1/2}\big),$$
where the implicit constant may depend on $n$ and $m$.
\end{cor}

\begin{proof}
Write $S\rd{^{\ii}} = S_0 + \sum_{k\neq j} S_k X_k$ and $R\rd{^{\ii}} = R_0 + \sum_{k\neq j} R_k X_k$, where $\forall$ $k\in\{1,\ldots,m\}\backslash\{j\}$, $S_k,R_k\in\ncp{n}$ depend neither on the $k$-th or the $j$-th variable, and $S_0,R_0\in M_n(\C)$. Then for any $k\neq j$, \rd{since $S=I^{\ii}S$ and $R^{*}=R^{*}I^{\ii}$,}
\begin{align*}
\Eghi\|S\rd{^{\ii}} G_k H_k H_{k}^* G_{k}^*(R_{k}\rd{^{\ii}})^*\| & =\rd{\Eghi\|I^{\ii}S^{\ii} G_k H_k H_{k}^* G_{k}^*(R_{k}^{\ii})^* I^{\ii}\|}\\
&\rd{\leq \frac{n}{\sqrt{p}}\Eghi\|S^{\ii}G_j\| \|G_{k}^*(R_{k}^{\ii})^*\|}\\
&\rd{\leq \frac{n}{\sqrt{p}}\Big(\Eghi\|S^{\ii}G_j\|^{2}\Big)^{1/2} \Big(\Eghi\|G_{k}^*(R_{k}^{\ii})^*\|^{2} \Big)^{1/2} }\\
&\rd{\leq \frac{n^{2}}{\sqrt{p}}\Big(\Eghi\|S^{\ii}\|^{2}\Big)^{1/2} \Big(\Eghi\|R_{k}^{\ii}\|^{2} \Big)^{1/2} }\\
\end{align*}
where the last inequality used
$\E_{G\sim\rd{\iota}(\mathcal{G})} \|G\|^2 \rd{= \E_{G\sim\mathcal{G}}\|G\|^{2}} \leq \E_{G\sim\mathcal{G}} \Tr(GG^*)=n$ and the first inequality used Lemma~\ref{lem:ev-bound}.  To conclude, we use the normalization assumption on $R,S$, and~\eqref{six30} \rd{and \eqref{five3}} from Lemma~\ref{lemma80} below.
\end{proof}

\begin{lemma}\label{lemma80}
Let $f\colon\{-1,1\}^{m}\to M_{n}(\C)$ and $Q\in\ncp{n}$ the multilinear polynomial associated to $f$.  Let $0\leq k\leq m$.  Let $G_{i}\sim\mathcal{G}$ and $b_{i}\sim\mathcal{B}$.  Let $\mathcal{X}=(G_{1},\ldots,G_{k})$ and $\mathcal{Y}=(b_{k+1},\ldots,b_{m})$.  Then
\begin{equation}\label{six30}
\E_{x_i\sim\mathcal{X},y_j\sim\mathcal{Y}}\mnormsq{Q\{x_i,y_j\}}
=\E_{b_{i}\sim\mathcal{B}}\mnormsq{Q\{b_{i}\}}.
\end{equation}
\begin{equation}\label{five1.5}
\mathrm{Inf}_{i}(Q^{\iota})=\mathrm{Inf}_{i}(Q),\quad\forall\,i\in\{1,\ldots,m\}.
\end{equation}
\begin{equation}\label{five2}
\frac{1}{n}\E_{b_{i}\sim\mathcal{B}}\Tr\mnormsq{Q^{\iota}\{b_{i}\}}
=\frac{1}{n}\E_{b_{i}\sim\mathcal{B}}\Tr\mnormsq{Q\{b_{i}\}}.
\end{equation}
\begin{equation}\label{five3}
\frac{1}{n}\Eghi\Tr\mnormsq{Q^{\iota}\{G_{i}H_{i}\}}
=\frac{1}{n}\E_{G_i\sim \mathcal{G}}\Tr\mnormsq{Q\{G_{i}\}}.
\end{equation}
\end{lemma}
\begin{proof}
Recall the variables $G_{i}\sim\mathcal{G}$ are independent, $\E G_{i}=0$ and $\E G_{i}G_{i}^{*}=I$ for all $1\leq i\leq k$.  Similarly, the $b_{i}\sim\mathcal{B}$ are independent with $\E b_{i}=0$ and $\E b_{i}b_{i}^{*}=1$ for all $k+1\leq i\leq m$.  So,
\begin{flalign*}
&\E_{x\sim\mathcal{X},y\sim\mathcal{Y}}(Q(x,y)(Q(x,y))^{*})\\
&\qquad=\E\sum_{S\subset\{1,\ldots,m\}}\Big(\widehat{Q}(S)\big(\prod_{i\in S\colon i\leq k} G_{i}\prod_{i\in S\colon i> k}b_{i}\big)
\big(\prod_{i\in S\colon i\leq k} G_{i}\prod_{i\in S\colon i> k}b_{i}\big)^{*}(\widehat{Q}(S))^{*}
\Big)\\
&\qquad=\sum_{S\subset\{1,\ldots,m\}}\widehat{Q}(S)(\widehat{Q}(S))^{*}\\
&\qquad\stackrel{\eqref{eqplan}}{=}\E_{b_{i}\sim\mathcal{B}}\mnormsq{Q\{b_{i}\}}.
\end{flalign*}
Then \eqref{six30} follows.  Equation~\eqref{five1.5} follows from \eqref{six0}, Definition \ref{def:embedding}, and Lemma \ref{lemma80}.  Equalities \eqref{five2} and \eqref{five3} follow from Definition \ref{def:embedding}.
\end{proof}

\subsection{Unique Games Conjecture}
\label{sec:ug}

The Unique Games Conjecture is a commonly assumed conjecture in complexity theory, though its current status is unresolved.

\begin{definition}[\embolden{Unique Games}]\label{max2lindef}
Let $m\in\N$.  Let $G=G(\mathcal{S},\mathcal{W},\mathcal{E})$ be a bipartite graph with vertex sets $\mathcal{S}$ and $\mathcal{W}$ and edge set $\mathcal{E}\subseteq \mathcal{S}\times\mathcal{W}$. For all $(v,w)\in \mathcal{E}$, let $\pi_{vw}\colon\{1,\ldots,m\}\to\{1,\ldots,m\}$ be a permutation.  An instance of the Unique Games problem is $\mathcal{L}=(G(\mathcal{S},\mathcal{W},\mathcal{E}),m,\{\pi_{vw}\}_{(v,w)\in \mathcal{E}})$. $m$ is called the alphabet size of $\mathcal{L}$. A labeling of $\mathcal{L}$ is a function $\eta\colon \mathcal{S}\cup \mathcal{W}\to\{1,\ldots,m\}$.  An edge $(v,w)\in \mathcal{E}$ is satisfied if and only if $\eta(v)=\pi_{vw}(\eta(w))$.  The goal of the Unique Games problem is to maximize the fraction of satisfied edges of the labeling, over all such labelings $\eta$.  We call the maximum possible fraction of satisfied edges $\mathrm{OPT}(\mathcal{L})$.
\end{definition}

\begin{definition}[\embolden{Unique Games Conjecture}, \cite{Khot02,khot07}]
For every $0<\beta<\alpha<1$ there exists $m=m(\alpha,\beta)\in\N$ and a family of Unique Games instances $(\mathcal{L}_n)_{n\geq 1}$ with alphabet size $m$ such that no polynomial time algorithm can distinguish between $\mathrm{OPT}(\mathcal{L}_n)<\beta$ or $\mathrm{OPT}(\mathcal{L}_n)>\alpha$.
\end{definition}

%============================%
\section{Majorization Principle}
\label{sec:invariance}
%============================%

\subsection{A noncommutative hypercontractive inequality}

One of the main tools used in the proof of our majorization principle is a hypercontractive inequality for noncommutative multilinear polynomials. The inequality bounds the $2K$-norm, for $K\geq 1$ an integer, of a polynomial $Q$ by the $2$-norm of $Q$ .  We refer to this inequality as a $(2K,2)$ hypercontractive inequality; it can be considered as a polynomial generalization of the noncommutative Khintchine inequality between the $2K$ norm and the $2$ norm. (The Khintchine inequality corresponds to the polynomial $Q(X_{1},\ldots,X_{m})=X_{1}$ and $G_{1}\sim\mathcal{V}$; see~\cite[Corollary 7.3]{mackey14} or~\cite{dirksen13}.)

Recall the definition of the ensembles $\mathcal{G}$  and $\mathcal{V}$ in Section~\ref{sec:random-variables}.

\begin{theorem}[\embolden{$(2K,2)$ Hypercontractivity}]\label{k2hyper}
Let $K\in\N$.  Let $Q$ be a noncommutative multilinear polynomial of degree $d\in\N$, as in~\eqref{eq:def-ncpoly}.  Let $G_{j}\sim\mathcal{G}$, where $\mathcal{G}$ is such that $\vnorm{\E (G_{1}G_{1}^{*})^{K}}\leq c_{K}$ for some $c_{K}\geq1$.  Then
\begin{flalign*}
\Eu\Tr\mnorm{Q\{G_{j}\}}^{2K}
\leq \hyconstv(\Eu\Tr\mnormsq{Q\{G_{j}\}})^{K}.
\end{flalign*}
Or, more generally,
\begin{flalign*}
\Eg\Tr\mnorm{Q\{G_{j}\}}^{2K}
\leq \hyconst(\Eg\Tr\mnormsq{Q\{G_{j}\}})^{K}.
\end{flalign*}
\end{theorem}
\begin{remark}\label{rkdis} As mentioned in \cite[Theorem 3.13]{mossel10} or \cite[Theorem 4]{nelson73}, the best possible constant in this hypercontractive inequality is $(2K-1)^{dK}$ in the case that $G_{j}\sim\mathcal{V}$ are replaced with $b_{i}\sim\mathcal{B}$. So, we achieve the optimal constant in this case, since we can use $c_{K}=1$ for all $K\in\N$ in this case.

The result that hypercontractivity also holds for the variables $G_{1}=b_{1}\sim\mathcal{B}$ generalizes a result of \cite[Lemma 6.1]{gross72}.
\end{remark}
\begin{proof}
We begin with the first inequality.  It suffices to prove the following hypercontractive estimate: if $\rho\geq0$ satisfies $\rho\leq(2K-1)^{-1/2}c_{K}^{-1/(2K)}$, then
\begin{equation}\label{newone}
\Eu\Tr\mnorm{T_{\rho}Q\{G_{j}\}}^{2K}
\leq (\Eu\Tr\mnormsq{Q\{G_{j}\}})^{K}.
\end{equation}
To see that~\eqref{newone} implies the first inequality of the theorem, choose $\rho=(2K-1)^{-1/2}c_{K}^{-1/(2K)}$, and observe that
\begin{flalign}
&\Eu\Tr\mnorm{Q\{G_{j}\}}^{2K}
\stackrel{\eqref{one32}}{=}\Eu\Tr\Big|T_{\rho}\Big(\sum_{S\subset\{1,\ldots,m\}\colon \abs{S}\leq d}\rho^{-\abs{S}}\widehat{Q}(S)\prod_{i\in S}G_{i}\Big)\Big|^{2K}\notag\\
&\stackrel{\eqref{newone}}{\leq} \Big(\Eu\Tr\Big|\sum_{S\subset\{1,\ldots,m\}\colon \abs{S}\leq d}\rho^{-\abs{S}}\widehat{Q}(S)\prod_{i\in S}G_{i}\Big|^{2}\Big)^{K}
\stackrel{\eqref{six30}\wedge\eqref{eqplan}}{=} \Big(\sum_{S\subset\{1,\ldots,m\}\colon \abs{S}\leq d}\rho^{-2\abs{S}}\Tr\absf{\widehat{Q}(S)}^{2}\Big)^{K}\label{new2}\\
&\leq \rho^{-2dK}\Big(\sum_{S\subset\{1,\ldots,m\}\colon \abs{S}\leq d}\Tr\absf{\widehat{Q}(S)}^{2}\Big)^{K}
\stackrel{\eqref{six30}\wedge\eqref{eqplan}}{=}(2K-1)^{dK}c_{K}^{d}(\Eu\Tr\absf{Q\{G_{j}\}}^{2})^{K}\label{new3}.
\end{flalign}
The proof of \eqref{newone} is by induction on the number $m$ of variables of $Q$. If $m=0$ then there are no variables and the inequality follows from the elementary inequality $\sum_{i=1}^{n}(\lambda_{i}(Q))^{2K}\leq(\sum_{i=1}^{n}(\lambda_{i}(Q))^{2})^{K}$ applied to the singular values of the (deterministic) matrix $Q$.  To establish the inductive step, write $Q=R_{0}+R_{1}X_{m}$, where $R_{0},R_{1}$ depend on at most $m-1$ variables each (for clarity we suppress this dependence from the notation).
%Note that $R_{0}$ has degree at most $d$, and $R_{1}$ has degree at most $d-1$.
Note that $T_{\rho}Q=T_{\rho}R_{0}+\rho (T_{\rho}R_{1})X_{m}$
We begin with a binomial expansion
\begin{equation}\label{aone99}
\mnorm{T_{\rho}Q(X_{1},\ldots,X_{m})}^{2K}\\
=\sum_{(a_{1},\ldots,a_{2K})\in\{0,1\}^{2K}}
\prod_{i=1}^{K}\rho^{a_{2i-1}}T_{\rho}R_{a_{2i-1}}X_{m}^{a_{2i-1}}(X_{m}^{*})^{a_{2i}}\rho^{a_{2i}}(T_{\rho}R_{a_{2i}})^{*}.
\end{equation}
Any term in the sum for which $a_{j}=0$ for an odd number of elements $j\in\{1,\ldots,2K\}$ has expectation zero. Applying H\"{o}lder's inequality,
\begin{equation}\label{extra}
\Eu\Tr\abs{T_{\rho}Q}^{2K}
\leq \E\Tr\abs{T_{\rho}R_{0}}^{2K}
+
\sum_{\substack{(a_{1},\ldots,a_{2K})\in\{0,1\}^{2K}\colon\\ \frac{1}{2}\sum_{i=1}^{2K}a_{i}\mathrm{\,is\,a\,positive\,integer}}}
\prod_{i=1}^{2K}\rho^{a_{i}}
\Big[\E\Tr \abs{(T_{\rho}R_{a_{i}})G_{m}^{a_{i}}}^{2K}\Big]^{\frac{1}{2K}}.
\end{equation}
Let $(a_{1},\ldots,a_{2K})\in\{0,1\}^{2K}$ such that $\ell\colonequals\frac{1}{2}\sum_{i=1}^{2K}a_{i}$ is a positive integer.  The number of times the term $\prod_{i=1}^{2K}\rho^{a_{i}}
\Big[\E\Tr \abs{(T_{\rho}R_{a_{i}})G_{m}^{a_{i}}}^{2K}\Big]^{\frac{1}{2K}}$ is repeated in the sum in \eqref{extra} is $\binom{2K}{2\ell}=\frac{(2K)!}{(2K-2\ell)!(2\ell)!}$.  That is, \eqref{extra} can be rewritten as
\begin{equation}\label{athree30c}
\Eu\Tr\abs{T_{\rho}Q}^{2K}
\leq \E\Tr\abs{T_{\rho}R_{0}}^{2K}
+
\sum_{\ell=1}^{K}\rho^{2\ell}\binom{2K}{2\ell}
\Big[\E\Tr \abs{(T_{\rho}R_{1})G_{m}}^{2K}\Big]^{\frac{\ell}{K}}
\Big[\E\Tr \abs{T_{\rho}R_{0}}^{2K}\Big]^{\frac{K-\ell}{K}}.
\end{equation}

For any $A,B\in M_{n}(\C)$ it holds that $\Tr\big|\abs{B}A^{*}A\abs{B}\big|^{K}\leq\Tr\big|\abs{B}^{K}(A^{*}A)^{K}\abs{B}^{K}\big|$ (see e.g.~\cite[Theorem IX.2.10]{bhatia97}), hence
\begin{equation}\label{athree84c}
\Tr(ABB^{*}A^{*})^{K}=\Tr(A^{*}ABB^{*})^{K}\leq\Tr((BB^{*})^{K}(A^{*}A)^{K}).
\end{equation}

Starting from~\eqref{athree30c} and applying \eqref{athree84c} to each of the inner terms,
\begin{flalign}
\Eu\Tr\abs{T_{\rho}Q}^{2K}
&\leq \E\Tr\abs{T_{\rho}R_{0}}^{2K}
+\sum_{\ell=1}^{K}\rho^{2\ell}\binom{2K}{2\ell}c_{K}^{\ell/K}
\Big[\E\Tr \abs{T_{\rho}R_{1}}^{2K}\Big]^{\frac{\ell}{K}}
\Big[\E\Tr \abs{T_{\rho}R_{0}}^{2K}\Big]^{\frac{K-\ell}{K}}\notag\\
&\leq [\E\Tr\abs{R_{0}}^{2}]^{K}
+\sum_{\ell=1}^{K}\rho^{2\ell}\binom{2K}{2\ell}c_{K}^{\ell/K}
\Big[\E\Tr \abs{R_{1}}^{2}\Big]^{\ell}
\Big[\E\Tr \abs{R_{0}}^{2}\Big]^{K-\ell},\label{athree30ac}
\end{flalign}
where the second inequality is obtained by applying the inductive hypothesis.  For any odd integer $J$ we denote $J!!=\prod_{i=0}^{(J-1)/2}(J-2i)$, or $J!!=1$ if $(J-1)/2<1$.  Now if $1\leq\ell\leq K$,
\begin{flalign*}
\binom{2K}{2\ell}\binom{K}{\ell}^{-1}
&=\frac{(2K)!}{K!}\frac{\ell!}{(2\ell)!}\frac{(K-\ell)!}{(2(K-\ell))!}
=\frac{2^{K}(2K-1)!!}{2^{\ell}(2\ell-1)!!2^{K-\ell}(2(K-\ell)-1)!!}\\
&=\frac{(2K-1)!!}{(2\ell-1)!!(2(K-\ell)-1)!!}
=\prod_{i=0}^{\ell-1}\frac{2K-2i-1}{2\ell-2i-1}
\leq(2K-1)^{\ell}.
\end{flalign*}
Using this inequality and $0\leq\rho\leq (2K-1)^{-1/2}c_{K}^{-1/(2K)}$ we get
$$\rho^{2\ell}\binom{2K}{2\ell}c_{K}^{\ell/K}
\leq\rho^{2\ell}(2K-1)^{\ell}c_{K}^{\ell/K}\binom{K}{\ell}
\leq\binom{K}{\ell},\qquad\forall\,\,1\leq \ell\leq K.$$
Applying this inequality to \eqref{athree30ac},
\begin{flalign*}
\Eu\Tr\abs{T_{\rho}Q}^{2K}
&\leq \sum_{\ell=0}^{K}\binom{K}{\ell}\Big[\Eu\Tr \abs{R_{1}}^{2}\Big]^{\ell}
\Big[\Eu\Tr \abs{R_{0}}^{2}\Big]^{K-\ell}\\
&=\big(\sum_{i=0}^{1}\Eu\Tr (R_{i}R_{i}^{*})\big)^{K}
=[\Eu\Tr(QQ^{*})]^{K},
\end{flalign*}
where the last equality follows from $\E G_{1}G_{1}^{*}= I$.

We now prove the second inequality of the Theorem.  We repeat the above proof, though now the sum in \eqref{extra} ranges over all $(a_{1},\ldots,a_{2K})\in\{0,1\}^{2K}$ such that $\sum_{i=1}^{2K}a_{i}\geq2$.  If $\ell>0$ satisfies $2\ell+1=\sum_{i=1}^{2K}a_{i}$, then the corresponding term appears $\binom{2K}{2\ell+1}$ times in \eqref{extra}, and we bound this term in \eqref{extra} with the arithmetic-mean geometric-mean inequality as follows.
\begin{flalign*}
\prod_{i=1}^{2K}\Big[\E\Tr \abs{(T_{\rho}R_{a_{i}})G_{m}^{a_{i}}}^{2K}\Big]^{\frac{1}{2K}}
&=\Big[\E\Tr \abs{(T_{\rho}R_{1})G_{m}}^{2K}\Big]^{\frac{2\ell+1}{2K}}\Big[\E\Tr \abs{T_{\rho}R_{0}}^{2K}\Big]^{\frac{2K-2\ell-1}{2K}}\\
&\leq\frac{1}{2}\Big[\E\Tr \abs{(T_{\rho}R_{1})G_{m}}^{2K}\Big]^{\frac{\ell}{K}}\Big[\E\Tr \abs{T_{\rho}R_{0}}^{2K}\Big]^{\frac{K-\ell}{K}}\\
&\qquad+\frac{1}{2}\Big[\E\Tr \abs{(T_{\rho}R_{1})G_{m}}^{2K}\Big]^{\frac{\ell+1}{K}}\Big[\E\Tr \abs{T_{\rho}R_{0}}^{2K}\Big]^{\frac{K-\ell-1}{K}}.
\end{flalign*}
Then, as above, we have $\frac{1}{2}\binom{2K}{2\ell+1}\binom{K}{\ell}^{-1}\leq\frac{K-1}{3}(2K-1)^{\ell}$ and $\frac{1}{2}\binom{2K}{2\ell+1}\binom{K}{\ell+1}^{-1}\leq\frac{2}{3}(2K-1)^{\ell}$.  So, if $0\leq\rho\leq (2K-1)c_{K}^{-1/(2K)}$, then
$\rho^{2\ell}\left(\binom{2K}{2\ell}+\frac{1}{2}\binom{2K}{2\ell+1}+\frac{1}{2}\binom{2K}{2\ell-1}\right)c_{K}^{\ell/K}
\leq\binom{K}{\ell},$ for all $1\leq \ell\leq K$, and then the proof concludes as above.

%%%%% put this here just for us:  (2K-1)!=(2K-1)!! 2^{K-1}(K-1)!   and (2\ell+1)!=(2\ell+1)!! 2^{\ell}\ell!, so
%\begin{flalign*}
%\frac{1}{2}\binom{2K}{2\ell+1}\binom{K}{\ell}^{-1}
%&=\frac{1}{2}\frac{(2K)!}{K!}\frac{\ell!}{(2\ell+1)!}\frac{(K-\ell)!}{(2(K-\ell)-1)!}\\
%&=\frac{2^{K-1}(2K-1)!!(K-\ell)!}{2^{\ell}(2\ell+1)!!2^{K-\ell-1}(2(K-\ell)-1)!!(K-\ell-1)!}\\
%&=\frac{(2K-1)!!(K-\ell)}{(2\ell+1)!!(2(K-\ell)-1)!!}
%=\frac{K-\ell}{2\ell+1}\prod_{i=0}^{\ell-1}\frac{2K-2i-1}{2\ell-2i-1}
%\leq\frac{K-1}{3}(2K-1)^{\ell}.
%\end{flalign*}
%\begin{flalign*}
%\frac{1}{2}\binom{2K}{2\ell+1}\binom{K}{\ell+1}^{-1}
%&=\frac{1}{2}\frac{(2K)!}{K!}\frac{(\ell+1)!}{(2\ell+1)!}\frac{(K-\ell-1)!}{(2(K-\ell)-1)!}
%=\frac{(2K-1)!!}{(2\ell+1)!!(2(K-\ell)-1)!!}\frac{(\ell+1)!}{\ell!}\\
%&=\frac{\ell+1}{2\ell+1}\prod_{i=0}^{\ell-1}\frac{2K-2i-1}{2\ell-2i-1}
%\leq\frac{2}{3}(2K-1)^{\ell}.
%\end{flalign*}

\end{proof}

\begin{cor}[\embolden{$(2K,2)$ Hypercontractivity}]\label{bigcor}
Let $K\in\N$.  Let $Q$ be a noncommutative multilinear polynomial of degree $d\in\N$, as in~\eqref{eq:def-ncpoly}. Then
\begin{flalign*}
\Eu\Tr\mnorm{Q\{G_{j}\}}^{2K}
\leq \hyconstv(K!)^{d}(\Eu\Tr\mnormsq{Q\{G_{j}\}})^{K}.
\end{flalign*}
$$
\Eivh\Tr\mnorm{Q^{\iota}\{G_{j}H_{j}\}}^{2K}
\leq \hyconstv(K!)^{d}(\Eivh
\Tr\mnormsq{Q^{\iota}\{G_{j}H_{j}\}})^{K}.
$$
\begin{flalign*}
\Eb\Tr\mnorm{Q\{b_{j}\}}^{2K}
\leq \hyconstv(\Eb\Tr\mnormsq{Q\{b_{j}\}})^{K}.
\end{flalign*}
\end{cor}

\begin{proof}
The first inequality follows from Theorem \ref{k2hyper} using~\cite[Corollary 2.8]{haagerup99} to show that for $G\sim\mathcal{V}$ and any $\ell\in\N$, $\vnorm{\E(GG^{*})^{\ell}}\leq \ell!$\;.

To prove the second inequality, we follow the proof of Theorem \ref{k2hyper} using $\vnorm{\E(GG^{*})^{\ell}}\leq \ell!$ for any $G\sim\mathcal{V}\otimes\Id$, where the $Q$ used in the proof becomes $Q^{\iota}$.  Writing $Q^{\iota}=R_{0}^{\iota}+R_{1}^{\iota}X_{m}$, there are only two required changes.  First, the equalities \eqref{new2} and \eqref{new3} are justified by combining \eqref{five3} with \eqref{six30} and \eqref{eqplan}.  For example, \eqref{new2} is justified by
\begin{flalign*}
&\Big(\Eivh\Tr\Big|\sum_{S\subset\{1,\ldots,m\}\colon \abs{S}\leq d}\rho^{-\abs{S}}\widehat{Q^{\iota}}(S)\prod_{i\in S}G_{i}^{\iota}H_{i}\Big|^{2}\Big)^{K}\\
&\qquad\stackrel{\eqref{five3}}{=}  \Big(\Eu\Tr\Big|\sum_{S\subset\{1,\ldots,m\}\colon \abs{S}\leq d}\rho^{-\abs{S}}\widehat{Q}(S)\prod_{i\in S}G_{i}\Big|^{2}\Big)^{K}\\
&\qquad\stackrel{\eqref{six30}\wedge\eqref{eqplan}}{=} \Big(\sum_{S\subset\{1,\ldots,m\}\colon \abs{S}\leq d}\rho^{-2\abs{S}}\Tr\absf{\widehat{Q}(S)}^{2}\Big)^{K}.
\end{flalign*}

And \eqref{new3} is justified in the same way.  Similarly, the last inequality in the proof is justified as
\begin{flalign*}
\Eivh\Tr\abs{T_{\rho}Q^{\iota}}^{2K}
&\leq \sum_{\ell=0}^{K}\binom{K}{\ell}\Big[\Eivh\Tr \abs{R_{1}^{\iota}}^{2}\Big]^{\ell}
\Big[\Eivh\Tr \abs{R_{0}^{\iota}}^{2}\Big]^{K-\ell}\\
&\stackrel{\eqref{five3}}{=}\sum_{\ell=0}^{K}\binom{K}{\ell}\Big[\Eu\Tr \abs{R_{1}}^{2}\Big]^{\ell}
\Big[\Eu\Tr \abs{R_{0}}^{2}\Big]^{K-\ell}\\
&=(\Eu \Tr\abs{Q}^{2})^{K}\\
&\stackrel{\eqref{five3}}{=}\Big(\Eivh \Tr\abs{Q\rd{^{\ii}}}^{2}\Big)^{K}.
\end{flalign*}

The last inequality in the Corollary follows directly from Theorem \ref{k2hyper}.
\end{proof}

In summary, $Q$ is hypercontractive when we substitute into $Q$ the noncommutative random variables $G_{i}H_{i}$.  Since $Q$ is also hypercontractive when we substitute into $Q$ commutative random variables distributed uniformly in $\{-1,1\}$, we get the standard consequence that $Q$ is hypercontractive when we substitute into it a mixture of commutative and noncommutative random variables.

\begin{cor}[\embolden{$(2K,2)$ Hypercontractivity for mixed inputs}]\label{lemma2.7c}
Let $G_{j}\sim\mathcal{G}$ be i.i.d. random $n\times n$ matrices and let $Q$ be a noncommutative multilinear polynomial of degree $d\in\N$ such that $Q$ satisfies $(2K,2)$ hypercontractivity for some $K\in\N$.  That is, assume there exists $c_{K}\geq1$ such that
$$\Eg\Tr\mnorm{Q\{G_{j}\}}^{2K}\leq \hyconst(\Eg\Tr\mnormsq{Q\{G_{j}\}})^{K}.$$
 Let $\mathcal{X}=(G_{1},\ldots,G_{k})$. Let $b_{j}\sim\mathcal{B}$ and let $\mathcal{Y}=(b_{k+1},\ldots,b_{m})$.  Then
$$ \E_{x\sim\mathcal{X},y\sim\mathcal{Y}}\Tr\mnorm{Q(x,y)}^{2K}
\leq \hyconst\big(\E_{x\sim\mathcal{X},y\sim\mathcal{Y}}\Tr\mnormsq{Q(x,y)}\big)^K.$$
\end{cor}
\begin{proof}
For any $p\geq1$, let $\vnorm{\cdot}_{p,\mathcal{Z}}$ denote the norm $\vnorm{Q}_{p,\mathcal{Z}}=(\E_{z\sim\mathcal{Z}}\mathrm{Tr}\abs{Q(z)}^{p})^{1/p}$.
\begin{flalign*}
\vnormf{Q}_{2K,\mathcal{X}\cup\mathcal{Y}}
&=\Big\|\sum_{S\subset\{1,\ldots,m\}}\widehat{Q}(S)\prod_{j\in S\colon j\leq k}G_{j}\prod_{j\in S\colon j>k}b_{j}\Big\|_{2K,\mathcal{X}\cup\mathcal{Y}}\\
&=\Big\|\,\Big\|\sum_{S\subset\{1,\ldots,m\}}\widehat{Q}(S)\prod_{j\in S\colon j\leq k}G_{j}\prod_{j\in S\colon j>k}b_{j}\Big\|_{2K,\mathcal{X}}\Big\|_{2K,\mathcal{Y}}\\
&\leq\hyconstmixedk
\Big\|\,\Big\|\sum_{S\subset\{1,\ldots,m\}}\Big(\widehat{Q}(S)\prod_{j\in S\colon j>k}b_{j}\Big)
\prod_{j\in S\colon j\leq k}G_{j}\Big\|_{2,\mathcal{X}}\Big\|_{2K,\mathcal{Y}},
\end{flalign*}
by Theorem \ref{k2hyper}.  Next, using Minkowski's inequality, from the above we get
\begin{flalign*}
\vnormf{Q}_{2K,\mathcal{X}\cup\mathcal{Y}}
&\leq\hyconstmixedk\Big\|\,\Big\|\sum_{S\subset\{1,\ldots,m\}}\Big(\widehat{Q}(S)
\prod_{j\in S\colon j>k}b_{j}\Big)
\prod_{j\in S\colon j\leq k}G_{j}\Big\|_{2K,\mathcal{Y}}\Big\|_{2,\mathcal{X}}\\
&=\hyconstmixedk\Big\|\,\Big\|\sum_{S\subset\{1,\ldots,m\}}\Big(\widehat{Q}(S)
\prod_{j\in S\colon j\leq k}G_{j}\Big)
\prod_{j\in S\colon j>k}b_{j}\Big\|_{2K,\mathcal{Y}}\Big\|_{2,\mathcal{X}}\\
&\leq \hyconstmixed\Big\|\,\Big\|\sum_{S\subset\{1,\ldots,m\}}\Big(\widehat{Q}(S)
\prod_{j\in S\colon j\leq k}G_{j}\Big)
\prod_{j\in S\colon j>k}b_{j}\Big\|_{2,\mathcal{Y}}\Big\|_{2,\mathcal{X}}\\
&=\hyconstmixed\Big\|\sum_{S\subset\{1,\ldots,m\}}\widehat{Q}(S)\prod_{j\in S\colon j\leq k}G_{j}
\prod_{j\in S\colon j>k}b_{j}\Big\|_{2,\mathcal{X}\cup\mathcal{Y}}\\
&=\hyconstmixed\vnormf{Q}_{2,\mathcal{X}\cup\mathcal{Y}},
\end{flalign*}
where the third line is by Theorem \ref{k2hyper}.
\end{proof}

%----------------------%
\subsection{Majorization principle}
\label{sec:inv-principle}
%----------------------%

\begin{theorem}[\embolden{Noncommutative Majorization Principle for Increasing Test Functions}]\label{thm:invariance-convex}
Let $Q\in\ncp{n}$ be a noncommutative multilinear polynomial of degree $d$ such that  $\Eb\frac{1}{n}\Tr(Q\{b_{j}\}Q\{b_{j}\}^{*}))\leq1$.  Let $\tau\colonequals\max_{1\leq j\leq m}\mathrm{Inf}_{j}Q$.  Let $G_{i}\sim\mathcal{V}\otimes\Id$.  Assume $\vnorm{\E (G_{1}G_{1}^{*})^{2}}\leq c_{2}$ and $\vnorm{\E (G_{1}G_{1}^{*})^{3}}\leq c_{3}$ with $c_{2},c_{3}\geq1$.

Let $\psi\colon[0,\infty)\to\R$ be a function with three continuous derivatives such that $\psi'(t)\geq0$ for all $t\geq0$.  Let $a_{2}=\sup_{t\geq0}\abs{\psi''(t)}$ and  $a_{3}=\sup_{t\geq0}\abs{\psi'''(t)}$.  Then
\begin{flalign*}
\Eghi\frac{1}{n}\mathrm{Tr}\psi|Q^{\iota}\{G_{i}H_{i}\}|^2
\leq\Eb\frac{1}{n}\mathrm{Tr}\psi|Q\{b_{j}\}|^2
+ a_{3} n^{3/2} \hyconstab\tau^{1/2}  + O_{n,m}(a_2 p^{-1/2}).
\end{flalign*}
\end{theorem}

\begin{proof}
We show the bound using the Lindeberg replacement method, replacing the $m$ variables $G_k H_k$ by $b_j$ one at a time, starting from the last, for each $j\in\{1,\ldots,m\}$.  Suppose variables $j+1,j+2,\ldots,m$ have already been replaced, and write $Q^{\iota}=R+SX_{j}$ where $R,S$ do not depend on the $j^{th}$ variable.

Any three times continuously differentiable $F\colon[0,\infty)\to\R$ has a Taylor expansion
\begin{equation}\label{five100}
F(1)=F(0)+F'(0)+\frac{1}{2}F''(0)+\frac{1}{2}\int_{0}^{1}(1-s)^{2}F'''(s)ds.
\end{equation}
Let $F(t)=\mathrm{Tr}\psi((R+tSX)(R+tSX)^{*})$ for $t\in[0,1]$.  Then
\begin{align}
F(0)&=\mathrm{Tr}\psi(RR^{*}),\label{five70}\\
F'(0)&=\mathrm{Tr}(\psi'(RR^{*})(SX_{}R^{*}+RX_{}^{*}S^*)),\label{five71}\\
F''(0)&=\mathrm{Tr}(\psi'(RR^{*})2SX_{}X_{}^{*}S^*+\psi''(RR^{*})(SX_{}R^{*}+RX_{}^{*}S^*)^{2}),\label{five72}\\
F'''(t)
&=\mathrm{Tr}(\psi''((R+tSX_{})(R+tSX_{})^{*})((SX_{}R^{*}+RX_{}^{*}S^*)+2tSX_{}X_{}^{*}S^*)(SX_{}X_{}^{*}S^*))\notag\\
&+\mathrm{Tr}(\psi''((R+tSX_{})(R+tSX_{})^{*})2((SX_{}R^{*}+RX_{}^{*}S^*)+2tSX_{}X_{}^{*}S^*)(SX_{}X_{}^{*}S^*))\notag\\
&+\mathrm{Tr}(\psi'''((R+tSX_{})(R+tSX_{})^{*})2((SX_{}R^{*}+RX_{}^{*}S^*)+2tSX_{}X_{}^{*}S^*)^{3}).\label{five73}
\end{align}
For any $t\in[0,1]$, let $F_{1}(t)=\mathrm{Tr}\psi((R+tSb_j)(R+tSb_j)^{*})$ and $F_{2}(t)=\mathrm{Tr}\psi((R+tSG_{j}H_{j})(R+tSG_{j}H_{j})^{*})$.  From~\eqref{five100},
\begin{equation}\label{eq:convex-1}
\E F_{2}(1)-\E F_{1}(1)
=\E F_{2}''(0)-\E F_{1}''(0)+\E\frac{1}{2}\int_{0}^{1}(1-s)^{2}F_{2}'''(s)ds-\E\frac{1}{2}\int_{0}^{1}(1-s)^{2}F_{1}'''(s)ds,
\end{equation}
where we used that $\E F_{2}(0)=\E F_{1}(0)$ and $\E F_{2}'(0)=\E F_{1}'(0)$. We bound the two differences on the right-hand side of~\eqref{eq:convex-1} separately.

For the first, using that $\psi'(RR^{*})$ is positive semidefinite the first term can be bounded as
\begin{align}
\E\Tr(\psi'(RR^{*})2SG_{j}H_{j}H_{j}^{*}G_{j}S^{*})
&= \E\Tr(\psi'(RR^{*})2S\begin{pmatrix} I & 0\\ 0 & 0\end{pmatrix}S^{*}) \notag\\
&\leq \E\Tr(\psi'(RR^{*})2SS^{*})\notag\\
&=\E\Tr(\psi'(RR^{*})2Sb_jb_j^*S^{*}).\label{five74}
\end{align}
The second term in \eqref{five72} is readily bounded using Corollary~\ref{cor:rs-ev-bound}, from which it follows that $\E \|SG_k H_kR^*+RH_j^*G_j^*S^*\|=O_{n,m}(p^{-1/2})$, and $|\psi''(t)|\leq a_2$ for all $t\geq0$. Combining the two bounds,
$$\E F_{2}''(0)-\E F_{1}''(0) \leq O_{n,m}\big(a_{2}p^{-1/2}\big).$$

For the second difference on the right-hand side of~\eqref{eq:convex-1}, there are two terms, corresponding to the first two lines of~\eqref{five73} and the third line respectively. For the first two terms we apply the Cauchy-Schwarz inequality, isolating the last factor $SXX^*S^*$ and using $|\psi''(t)|\leq a_2$ for all $t\geq0$ to bound them by
\begin{align*}
&3a_2\Big(\E \Tr\big((SXX^*S^*)^2\big)\Big)^{1/2}\Big(\E\Tr\big( (SXR^*+RX^*S^*+2tSXX^*S^*)^2\big)\Big)^{1/2}\\
&\,\,\leq 3a_2\hyconsta\Big(\E \Tr(SS^*)\Big)\Big(6\E\Tr\big( SXR^*RX^*S^*\big)+12t^2\E\Tr\big((SXX^*S^*)^2\big)\Big)^{1/2}\\
&\,\,\leq3a_2 \hyconsta \Big(\E \Tr(SS^*)\Big) \Big(\E\Tr((SXX^*S^*)^2)(36\E\Tr((RR^*)^2+12^2t^2\E\Tr((SXX^*S^*)^2))\Big)^{1/4}\\
&\,\,\leq 36 a_2 \hyconsta n^{1/2} \Big(\E \Tr(SS^*)\Big)^{3/2},
\end{align*}
where for the first inequality, the first term is bounded using Corollary \ref{lemma2.7c} (first using Corollary \ref{bigcor} to show that hypercontractivity holds for $Q^{\iota}$ and then applying Corollary \ref{lemma2.7c} with $Q=Q^{\iota}$) and $\E XX^* \leq \Id$, and the second term is bounded using $(A+B+C)(A+B+C)^*\leq 4(AA^*+BB^*+CC^*)$, the second inequality uses Cauchy-Schwarz, and the last again Corollary \ref{lemma2.7c}, $\E XX^* \leq \Id$, and $\E RR^*\leq\Id$.

Finally we turn to the second term which appears in the expansion of the second difference on the right-hand side of~\eqref{eq:convex-1} according to~\eqref{five73}, corresponding to the third line of~\eqref{five73}. Letting $P=SXR^*+RX^*S^*+2tSXX^*S^*$, the term can be bounded using H\"older's inequality by
\begin{align*}
a_3 \E \Tr |P|^{3}
&= O(a_3)\,\big(\E \Tr \big|RXS^*\big|^3 + \E\Tr\big|SXX^*S^*\big|^3\big)\\
&= O(a_3)\big(\E \Tr |SX|^6 \big)^{1/2}\Big(\big(\E\Tr |R|^6\big)^{1/2}+\big( \E\Tr |SX|^6\big)^{1/2}\Big)\\
&= O(a_3) n^{3/2} \hyconstb\big(\E \Tr \big|S|^2 \big)^{3/2},
\end{align*}
where the last line uses Corollary~\ref{lemma2.7c} (applied as above) and $\E SS^*\leq \Id$, $\E RR^* \leq \Id$.

Combining all error estimates and using $\E \Tr SS^* = \Inf_j(Q)$ we obtain
\begin{flalign}
&\E \frac{1}{n}F_{2}(1)\nonumber\\
&\leq \E \frac{1}{n}F_{1}(1)+O_{n,m}(p^{-1/2})
+\E\frac{1}{2n}\int_{0}^{1}(1-s)^{2}|F_{2}'''(s)|ds+\E\frac{1}{2n}\int_{0}^{1}(1-s)^{2}|F_{1}'''(s)|ds\nonumber\\
&\leq \E \frac{1}{n}F_{1}(1)+O_{n,m}(a_{2}p^{-1/2})
+a_{3} n^{1/2} \hyconstb(\Inf_{j}(Q))^{3/2}.\label{five101}
\end{flalign}

Iterating over all $m$ variables,
\begin{align*}
&\Eghi\frac{1}{n}\Tr\psi\abs{Q^{\iota}\{G_iH_i\}}^{2} -\Ebi\frac{1}{n}\Tr\psi\abs{Q\{b_i\}}^{2}\\
 &\qquad\qquad\leq a_{3} n^{1/2} \hyconstab \big(\max_{1\leq j\leq m} \Inf_j(Q)\big)^{1/2} \Big(\sum_{j=1}^m \Inf_j(Q)\Big) + O_{n,m}(p^{-1/2})\\
&\qquad\qquad\leq a_{3} n^{3/2} \hyconstab\big(\max_{1\leq j\leq m} \Inf_j(Q)\big)^{1/2}  + O_{n,m}(a_{2}p^{-1/2}).
\end{align*}
\end{proof}

Let $\psi\colon\R\to\R$ be Lipschitz, so that $\sup_{x\neq y\in\R}\frac{\abs{\psi(x)-\psi(y)}}{\abs{x-y}}\leq1$.  Let $x\in\R$ and let $\phi(x)=e^{-x^{2}/2}/\sqrt{2\pi}$, $\phi\colon\R\to\R$.  For any $\lambda>0$, define $\phi_{\lambda}(x)=\lambda^{-1}\phi(x/\lambda)$.  Define
\begin{equation}\label{eq:def-psilambda}
\psi_{\lambda}(x)=\psi*\phi_{\lambda}(x)=\int_{\R}\psi(t)\phi_{\lambda}(x-t)dt.
\end{equation}
Then $\abs{\psi(x)-\psi_{\lambda}(x)}<\lambda$ for all $x\in\R$, and $\absf{\frac{d^{k}}{dx^{k}}\phi_{\lambda}(x)}\leq\lambda^{1-k}$ for all $x\in\R$, so that
$$\Big|\frac{d^{k}}{dx^{k}}\psi_{\lambda}(x)\Big|\leq3\lambda^{1-k},\qquad\forall\,\,x\in\R,\qquad\forall\,1\leq k\leq 3.$$

\begin{lemma}\label{lemma70}
Let $\lambda>0$.  If $\psi$ is convex, then $\psi_{\lambda}$ is convex, and $\psi(x)\leq\psi_{\lambda}(x)$ for all $x\in\R$.
\end{lemma}

\begin{proof}
The first property is a standard differentiation argument for convolutions.  Since $\psi(x+h)+\psi(x-h)-2\psi(x)\geq0$ for all $x,h\in\R$, we also have $\psi_{\lambda}(x+h)+\psi_{\lambda}(x-h)-2\psi_{\lambda}(x)\geq0$.  The second property follows from Jensen's inequality.
\end{proof}

Let $\psi\colon\R\to\R$ be defined by
\begin{equation}\label{eq:def-psi}
\psi(t)\,=\,\max\big(0,\abs{t}-1\big), \qquad\forall t\in\R.
\end{equation}

\begin{theorem}\label{thm:invariance}
Let $Q\in\ncp{n}$ be a noncommutative multilinear polynomial of degree $d$ with $\Eb \frac{1}{n}\Tr( \abs{Q\{b_j\}}^{2} ) \leq 1$, and let $\tau = \max_{i=1,\ldots,m}\mathrm{Inf}_{i}(Q)$.   Let $G_{i}\sim\mathcal{G}\otimes\Id$.  Assume $\vnorm{\E (G_{1}G_{1}^{*})^{2}}\leq c_{2}$ and $\vnorm{\E (G_{1}G_{1}^{*})^{3}}\leq c_{3}$ with $c_{2},c_{3}\geq1$. Let $\psi$ be as in~\eqref{eq:def-psi}. Then
\begin{align*}
\Eghi\frac{1}{n}\Tr\psi |Q^{\ii}\{G_{i} H_{i}\}|^{2}
\leq\Eb\frac{1}{n}\Tr\psi |Q\{b_j\}|^{2}
+n^{1/2}\hyconstab\tau^{1/6}+O_{m,n}(\tau^{-1/3}p^{-1/2}).
\end{align*}
\end{theorem}

\begin{proof}
Let $\lambda>0$, and define $\psi_\lambda$ as in~\eqref{eq:def-psilambda}.  From Lemma \ref{lemma70}, $\psi_{\lambda}(x)\geq\psi(x)\geq0$ for all $x\in\R$.  So,
\begin{equation}\label{five200}
\Eghi\frac{1}{n}\mathrm{Tr}\psi\abs{Q^{\iota}\{G_{i}H_{i}\}}^{2}
\leq \Eghi\frac{1}{n}\mathrm{Tr}\psi_{\lambda}\abs{Q^{\iota}\{G_{i}H_{i}\}}^{2}.
\end{equation}
From Theorem \ref{thm:invariance-convex},
\begin{equation}\label{five201}
\begin{aligned}
\Eghi\frac{1}{n}\mathrm{Tr}\psi_{\lambda}\abs{Q^{\iota}\{G_{i}H_{i}\}}^{2}
&\leq\E_{b_{i}\sim\mathcal{B}}\frac{1}{n}\mathrm{Tr}\psi_{\lambda}\abs{Q\{b_{i}\}}^{2}\\
&\qquad+\lambda^{-2}\left( O(n^{3/2})\hyconstab\tau^{1/2} + O_{n,m}(p^{-1/2})\right).
\end{aligned}
\end{equation}
Using $\psi_{\lambda}(x)\leq\psi(x)+\lambda$ for all $x\geq0$,
\begin{equation}\label{five202}
\begin{aligned}
\frac{1}{n}\Tr\psi_{\lambda}\abs{Q(\sigma)}^{2}\leq \frac{1}{n}\Tr\psi\abs{Q(\sigma)}^{2}+\lambda,\qquad\forall\,\,\sigma\in\{-1,1\}^{m}.
\end{aligned}
\end{equation}
Combining \eqref{five200}, \eqref{five201} and \eqref{five202} completes the proof, with a choice of $\lambda$ such that $\lambda^3 = \Theta( n^{3/2}\tau^{1/2})$.
\end{proof}

Recall the definition of $T_\rho$ in~\eqref{one32} and the function $\mathrm{Chop}:\R\to\R$, $\mathrm{Chop}(t) = t$ if $|t|\leq 1$, $\mathrm{Chop}(t)=1$ if $t\geq 1$, and $\mathrm{Chop}(t)=-1$ if $t\leq -1$.

\begin{cor}[\embolden{Smoothed Version of Theorem \ref{thm:invariance}}]\label{cor:smooth-invariance}
Suppose $f\colon\cxhyper^{m}\to M_{n}(\C)$ with $\vnorm{f(\sigma)}\leq1$ for all $\sigma\in\{-1,1\}^{m}$.  Let $0<\rho<1$ and let $\tau\colonequals\max_{i=1,\ldots,m}\Inf_{i}f$.  Assume $\vnorm{\E (G_{1}G_{1}^{*})^{2}}\leq {c_{2}}$ and $\vnorm{\E (G_{1}G_{1}^{*})^{3}}\leq {c_{3}}$ for some $c_{2},c_{3}\geq1$.  Then $\vnorm{Q_{f}}_{L_{2},\mathcal{G}}\leq1$ and
\begin{equation}\label{three15}
\frac{1}{n}\vnorm{T_{\rho}Q_{f}^{\iota}-\mathrm{Chop}T_{\rho}Q_{f}^{\iota}}_{L_{2},\mathcal{G}}^{2}\leq 10n^{1/2}\tau^{\frac{1-\rho}{30(c_{2}c_{3})}}
+O_{m,n}(\tau^{-1/3}p^{-1/2}).
\end{equation}
\end{cor}%

\begin{proof}
Using the elementary inequality $[\max(0,t-1)]^{2}\leq \psi(t^{2})$ for all $t\geq0$, where $\psi$ is defined in~\eqref{eq:def-psi}, applied to the singular values of $T_{\rho}Q_{f}^{\iota}$,
\begin{equation}\label{five53.5}
\frac{1}{n}\vnorm{T_{\rho}Q_{f}^{\iota}-\mathrm{Chop}T_{\rho}Q_{f}^{\iota}}_{L_{2},\mathcal{G}}^{2}
\leq \Eghi\frac{1}{n}\Tr\psi\abs{T_{\rho}Q_{f}^{\iota}\{G_{i}H_{i}\}}^{2}.
\end{equation}
We first apply Theorem \ref{thm:invariance} to $P_{\leq d}(T_{\rho}Q_{f}^{\iota})$, where $d\in\N$ is to be determined later.
Since $0<\rho<1$, \eqref{one32}, \eqref{six0} imply that
$$\max_{i=1,\ldots,m}\mathrm{Inf}_{i}P_{\leq d}T_{\rho}Q_{f}
\,\leq\,\max_{i=1,\ldots,m}\Inf_{i}P_{\leq d}Q_{f}
\,\leq\,\max_{i=1,\ldots,m}\Inf_{i}Q_{f},$$
and we get by Theorem \ref{thm:invariance}  \rd{and \eqref{five1.5}}
\begin{align}
\Eghi\frac{1}{n}\Tr\psi |P_{\leq d}T_{\rho}Q_f^{\ii}\{G_{i} H_{i}\}|^{2}
&\leq\E_{b_{i}\sim\mathcal{B}}\frac{1}{n}\mathrm{Tr}\psi|P_{\leq d}T_{\rho}Q_f\{b_{i}\}|^{2}\notag\\
&\qquad +n^{1/2}\hyconstab\tau^{1/6}+O_{m,n}(\tau^{-1/3}p^{-1/2}).\label{nine5}
\end{align}

For any $a,b\in\R$, $\abs{\psi((a+b)^{2})-\psi(a^{2})}\leq 2\abs{a}\abs{b}+2\abs{b}^{2}$ follows by $\abs{\psi((a+b)^{2})-\psi(a^{2})}\leq\abs{b}\max_{t\in[a,a+b]}\abs{\frac{d}{dt}\psi(t^{2})}\leq\abs{b}2(\abs{a}+\abs{b})$.  Combining with the Cauchy-Schwarz inequality,
 \begin{align*}
\Tr\psi\abs{T_{\rho}Q_{f}^{\iota}}^{2}-\Tr\psi\abs{P_{\leq d}T_{\rho}Q_{f}^{\iota}}^{2}
&=\Tr\psi\abs{T_{\rho}P_{\leq d}Q_{f}^{\iota}+T_{\rho}P_{> d}Q_{f}^{\iota}}^{2}-\Tr\psi\abs{T_{\rho}P_{\leq d}Q_{f}^{\iota}}^{2}\\
&\leq2(\Tr |T_{\rho}P_{\leq d}Q_{f}^{\iota}|^2)^{1/2}(\Tr |T_{\rho}P_{> d}Q_{f}^{\iota}|^2)^{1/2} +2\Tr |T_{\rho} P_{> d}Q_{f}^{\iota}|^{2}.
\end{align*}
Taking expectation values and using \eqref{five3}, \eqref{six30} and \eqref{one32}, which imply that
\begin{align*}
\Eivh\Tr |T_{\rho} P_{> d}Q_{f}^{\iota}|^{2}&=\Eg\Tr |T_{\rho} P_{> d}Q_{f}^{\iota}|^{2}\\
&=\Eb\Tr |T_{\rho} P_{> d}Q_{f}^{\iota}|^{2}\\
&=\sum_{S\subset\{1,\ldots,m\}\colon\abs{S}> d}\rho^{2\abs{S}}\Tr(\widehat{f}(S)(\widehat{f}(S))^{*})\\
&\leq\rho^{2d}\sum_{S\subset\{1,\ldots,m\}}\Tr(\widehat{f}(S)(\widehat{f}(S))^{*})\\
&\leq n\rho^{2d},
\end{align*}
we get
\begin{equation}\label{nine6}
\frac{1}{n}\abs{\E_{b_{i}\sim\mathcal{B}}\left(\Tr\psi\abs{T_{\rho}Q_{f}^{\iota}\{b_{i}\}}^{2}-\Tr\psi\abs{P_{\leq d}T_{\rho}Q_{f}^{\iota}\{b_{i}\}}^{2}\right)}\leq 4\rho^{d},
\end{equation}
\begin{equation}\label{nine7}
\frac{1}{n}\absf{\Eghi\left(\Tr\psi\abs{T_{\rho}Q_{f}^{\iota}\{G_{i}H_{i}\}}^{2}-\Tr\psi\abs{P_{\leq d}T_{\rho}Q_{f}^{\iota}\{G_{i}H_{i}\}}^{2}\right)}\leq 4\rho^{d}.
\end{equation}
Using $t\leq\abs{t}$ for any $t\in\R$,
\begin{equation}\label{nine8}
\begin{aligned}
&\frac{1}{n}\Big(\Eghi\Tr\psi\abs{T_{\rho}Q_{f}^{\iota}\{G_{i}H_{i}\}}^{2}
-\E_{b_{i}\sim\mathcal{B}}\Tr\psi\abs{T_{\rho}Q_{f}^{\iota}\{b_{i}\}}^{2}\Big)\\
&\quad\leq\frac{1}{n}\Big|\Eghi\Tr\psi\abs{T_{\rho}Q_{f}^{\iota}\{G_{i}H_{i}\}}^{2} -\Eghi\Tr\psi\abs{T_{\rho}P_{\leq d}Q_{f}^{\iota}\{G_{i}H_{i}\}}^{2} \Big|\\
&\qquad\quad+\frac{1}{n}\Big(\Eghi\Tr\psi\abs{T_{\rho}P_{\leq d}Q_{f}^{\iota}\{G_{i}H_{i}\}}^{2}
-\E_{b_{i}\sim\mathcal{B}}\Tr\psi\abs{T_{\rho}P_{\leq d}Q_{f}^{\iota}\{b_{i}\}}^{2} \Big)   \\
&\qquad\quad+\frac{1}{n}\Big|\E_{b_{i}\sim\mathcal{B}}\Tr\psi\abs{T_{\rho}P_{\leq d}Q_{f}^{\iota}\{b_{i}\}}^{2}
 -\E_{b_{i}\sim\mathcal{B}}\Tr\psi\abs{T_{\rho}Q_{f}^{\iota}\{b_{i}\}}^{2}\Big| \\
&\leq n^{1/2}\hyconstab\tau^{1/6}+8\rho^{d} + O_{m,n}(\tau^{-1/3}p^{-1/2}),
\end{aligned}
\end{equation}
where the last inequality uses~\eqref{nine7} to bound the first term,~\eqref{nine5} for the second and~\eqref{nine6} for the third. From \eqref{one30} and \eqref{one32} we get $\vnorm{T_{\rho}Q_{f}(\sigma)}\leq1$ for all $\sigma\in\{-1,1\}^{m}$, so by definition of $\psi$ we have $\E_{b_{i}\sim\mathcal{B}}\Tr\psi\abs{T_{\rho}Q_{f}^{\iota}\{b_i\} }^{2}=0$. Combining~\eqref{five53.5} with~\eqref{nine8} and choosing $d=\min(\max(1,-\frac{\log(\tau)}{30(c_{2}c_{3})}),m)$ completes the proof (using $-\log\rho\geq1-\rho$ for all $0<\rho<1$).
\end{proof}

%----------------------%
\subsection{Moment Majorization}
\label{sec:imom-maj}
%----------------------%

Theorem \ref{thm:invariance-convex} implies that the even moments of a noncommutative multilinear polynomial follow a majorization principle.  Although we will not make use of Theorem \ref{invlemAb} for the applications in this paper, we include it as the statement could be of independent interest; the theorem has analogues in both the commutative \cite{mossel10} and free probability settings \cite{deya14}.

\begin{theorem}[\embolden{Noncommutative Majorization Principle for $2K^{th}$ Moments}]\label{invlemAb}
 Let $Q$ be a noncommutative multilinear polynomial of degree $d$ in $m$ variables, as in~\eqref{eq:def-ncpoly}. Suppose $\vnorm{Q(\sigma)}\leq1$ for all $\sigma\in\{-1,1\}^{m}$.  Let $p>n$, and let $Q^{\iota}$ be the zero-padded extension of $Q$, as defined in Definition \ref{def:embedding}.  Let $\tau\colonequals\max_{i=1,\ldots,m}\mathrm{Inf}_{i}Q$.  Let $G_{i}\sim\mathcal{G}$.  Assume  that $\vnorm{\E (G_{1}G_{1}^{*})^{K}}\leq c_{K}$, for some $K\in\N$ and $c_{K}\geq1$.  Then
\begin{equation}\label{eq:qboundK}
\begin{aligned}
&\Eghi\frac{1}{n}\Tr\mnorm{Q^{\iota}\{G_{i}H_{i}\}}^{2K}\\
&\qquad\leq\E_{b_{i}\sim\mathcal{B}}\frac{1}{n}\Tr\mnorm{Q\{b_{i}\}}^{2K}
+K^{3}\hyconst n^{2K}\tau^{1/4}+O_{m,n}(p^{-1/2}K^{2}\tau^{-1/4}).
\end{aligned}
\end{equation}
\end{theorem}

\begin{proof}
We begin with an upper tail estimate for $Q^{\iota}$.  From Markov's inequality,
\begin{equation}\label{eleven1}
\P_{\substack{G_{i}\sim\mathcal{G}\otimes\Id\\ H_{i}\sim\mathcal{H}_{p}}}\big(\Tr\abs{Q^{\iota}}^{2}>t\big)
=\P_{\substack{G_{i}\sim\mathcal{G}\otimes\Id\\ H_{i}\sim\mathcal{H}_{p}}}\big(\big(\Tr\abs{Q^{\iota}}^{2}\big)^{K}>t^{K}\big)
\leq t^{-K}\Eghi\big[\Tr\abs{Q^{\iota}}^{2}\big]^{K}.
\end{equation}
 Since $Q^{\iota}=I^{\iota}Q^{\iota}$ (where here $I$ denotes the $n\times n$ identity matrix), H\"{o}lder's inequality implies
\begin{equation}\label{eleven2}
\Eghi\big(\Tr\abs{Q^{\iota}}^{2}\big)^{K}
=\Eghi\big(\Tr\abs{I^{\iota}Q^{\iota}}^{2}\big)^{K} %use p=K, 1/p+1/q=1, so 1/q=1-1/K=[K-1]/K;  q = K/(K-1)
\leq n^{K-1}\Eghi\Tr\abs{Q^{\iota}}^{2K}.
\end{equation}
Combining \eqref{eleven1} and \eqref{eleven2} and applying Theorem \ref{k2hyper},
\begin{align}
\P_{\substack{G_{i}\sim\mathcal{G}\otimes\Id\\ H_{i}\sim\mathcal{H}_{p}}}\big(\Tr\abs{Q^{\iota}}^{2}>t\big)
&\leq t^{-K}n^{K-1}\hyconst\Big(\Eghi\Tr\abs{Q^{\iota}}^{2}\Big)^{K}\notag\\
&\leq t^{-K}n^{2K-1}\hyconst.\label{eleven3}
\end{align}

Let $s>0$ be a constant to be fixed later.  Define $\psi\colon[0,\infty)\to[0,\infty)$ so that $\psi(t)=t^{K}$ for any $0\leq t\leq s$, and $\psi$ is linear with slope $K s^{K-1}$ on $(s+1,\infty)$.  It is possible to construct such a $\psi$ with all three derivatives bounded, and in particular the third derivative bounded by some $a_3\leq K^{3}s^{K-3}$ on the interval $[s, s+1]$.  From Theorem \ref{thm:invariance-convex},
\begin{equation}\label{ten5}
\begin{aligned}
&\Eghi \frac{1}{n}\Tr\psi\abs{Q^{\iota}\{G_{i}H_{i}\}}^{2}\\
&\qquad\leq\Ebi\frac{1}{n}\Tr\psi\abs{Q\{b_{i}\}}^{2}+K^{3}s^{K-3}n^{3/2}\hyconstab\tau^{1/2}+O_{n,m}(p^{-1/2}K^{2}s^{K-2}).
\end{aligned}
\end{equation}

Finally, defining $a=K^{3}s^{K-3}n^{3/2}\hyconstab\tau^{1/2}+O_{n,m}(p^{-1/2}K^{2}s^{K-2})$ and letting $D$ be the event that $\Tr\psi\abs{Q^{\iota}\{G_{i}H_{i}\}}^{2}\leq s$, we have
\begin{flalign*}
\Eghi \frac{1}{n}\Tr\abs{Q^{\iota}\{G_{i}H_{i}\}}^{2K}
&=\Eghi \frac{1}{n}\Tr\abs{Q^{\iota}\{G_{i}H_{i}\}}^{2K}\cdot1_{D}+\Eghi \frac{1}{n}\Tr\abs{Q^{\iota}\{G_{i}H_{i}\}}^{2K}\cdot1_{D^{c}}\\
&=\Eghi \frac{1}{n}\Tr\psi\abs{Q^{\iota}\{G_{i}H_{i}\}}^{2}+\Eghi \frac{1}{n}\Tr\abs{Q^{\iota}\{G_{i}H_{i}\}}^{2K}\cdot1_{D^{c}}.
\end{flalign*}
Using~\eqref{ten5} to bound the first term and Cauchy-Schwarz for the second, the above can be bounded as
\begin{align*}
\Eghi &\frac{1}{n}\Tr\abs{Q^{\iota}\{G_{i}H_{i}\}}^{2K}\\
&\leq\Ebi\frac{1}{n}\Tr\psi\abs{Q\{b_{i}\}}^{2}+a+\Big(\Eghi \Big(\frac{1}{n}\Tr\abs{Q^{\iota}\{G_{i}H_{i}\}}^{2K}\Big)^{2}\Big)^{1/2}\Big(\Eghi 1_{D^{c}}\Big)^{1/2}\\
&\leq\Ebi\frac{1}{n}\Tr\psi\abs{Q\{b_{i}\}}^{2}+a+\Big(\Eghi \Tr\abs{Q^{\iota}\{G_{i}H_{i}\}}^{4K}\Big)^{1/2}
\Big(\P_{\substack{G_{i}\sim\mathcal{G}\otimes\Id\\ H_{i}\sim\mathcal{H}_{p}}}\big(\Tr\abs{Q^{\iota}}^{2}>s\big)\Big)^{1/2}\\
&\leq\Ebi\frac{1}{n}\Tr\psi\abs{Q\{b_{i}\}}^{2}+a+\hyconst\big(\Eghi \Tr\abs{Q^{\iota}\{G_{i}H_{i}\}}^{2}\big)^{K}
s^{-K/2}n^{K-1},
\end{align*}
using Theorem \ref{k2hyper} to bound the first term inside a square root, and~\eqref{eleven3} for the second. Finally, using Lemma~\ref{lemma80} and $\vnorm{Q(\sigma)}\leq1$ for all $\sigma\in\{-1,1\}^{m}$,  $\Eghi \Tr\abs{Q^{\iota}\{G_{i}H_{i}\}}^{2} \leq 1$.  Choosing $s=\tau^{-1/(4K)}$ finishes the proof.
\end{proof}

\section{Maximizing Noncommutative Noise Stability}\label{secnoise}

By adapting the proof of the Majority is Stablest Theorem from \cite{mossel10}, we can get the following consequence of Corollary \ref{cor:smooth-invariance}.
\begin{cor}\label{cor3}
Let $0\leq\rho<1$ and let $\epsilon>0$.  Let $\delta=20n^{1/2}\tau^{\frac{1-\rho}{30(c_{2}c_{3})}}
+O_{m,n}(\tau^{-1/3}p^{-1/2})$.  Then there exists $\tau>0$ such that, if $f\colon\{-1,1\}^{m}\to M_{n}(\C)$ satisfies $\vnorm{f(\sigma)}\leq1$ for all $\sigma\in\{-1,1\}^{m}$, $\E_{b_{i}\sim\mathcal{B}} f\{b_i\}=0$, and $
\max_{i=1,\ldots,m}\mathrm{Inf}_{i}(f)<\tau$, then
\begin{equation}\label{twelve0}
\frac{1}{n}\Ebi\Tr\abs{T_{\rho}Q_{f}\{b_{i}\}}^{2}
\leq \frac{1}{n}\Eghi\Tr\abs{\mathrm{Chop}T_{\rho}Q_{f}^{\iota}\{G_{i}H_{i}\}}^{2}+O(\epsilon+\delta).
\end{equation}
Moreover, $\absf{\frac{1}{n}\Eghi\Tr\,\mathrm{Chop}T_{\rho}Q_{f}^{\iota}\{G_{i}H_{i}\}}\leq\delta$.
\end{cor}

\begin{remark}
In the case $n=p=1$, it is known from \cite{mossel10} that the right-hand side of \eqref{twelve0} is $\frac{2}{\pi}\arcsin\rho +\epsilon$.  For larger $n$, the left side of \eqref{twelve0} can be interpreted as the noise stability of $Q_{f}$ with discrete inputs, and the right side as the noise stability of a function with operator norm pointwise bounded by $1$ under random Gaussian matrix inputs.  Eq. \eqref{twelve0} can thus be thought of as a matrix-valued version of one of the two main steps in the proof of the Majority is Stablest Theorem. However, for larger $n$, there seems to be no version of Borell's isoperimetric inequality that describes what the right-hand side of \eqref{twelve0} should be.  (Recall that Borell's isoperimetric inequality states that the noise stability of a subset of Euclidean space of fixed Gaussian measure is maximized when the set is a half space.)
\end{remark}

\begin{proof}
Since $\Ebi f\{b_{i}\}=0$, we have $\Ebi Q_{f}\{b_{i}\}=0$.  Using the Cauchy-Schwarz inequality and Corollary \ref{cor:smooth-invariance},
\begin{flalign*}
&\Big|\frac{1}{n}\Eghi\Tr\abs{T_{\rho}Q_{f}^{\iota}\{G_{i}H_{i}\}}^{2}-\frac{1}{n}\Eghi\Tr\abs{\mathrm{Chop}T_{\rho}Q_{f}^{\iota}\{G_{i}H_{i}\}}^{2}\Big|\\
&=\Big|\frac{1}{n}\Eghi\Tr\abs{T_{\rho}Q_{f}^{\iota}\{G_{i}H_{i}\}}^{2}-\frac{1}{n}\Eghi\Tr\Big(\mathrm{Chop}T_{\rho}Q_{f}^{\iota}\{G_{i}H_{i}\} [T_{\rho}Q_{f}^{\iota}\{G_{i}H_{i}\}]^{*}\Big)\Big|\\
&\quad+\Big|\frac{1}{n}\Eghi\Tr\Big(\mathrm{Chop}T_{\rho}Q_{f}^{\iota}\{G_{i}H_{i}\} [T_{\rho}Q_{f}^{\iota}\{G_{i}H_{i}\}]^{*}\Big)
-\frac{1}{n}\Eghi\Tr\abs{\mathrm{Chop}T_{\rho}Q_{f}^{\iota}\{G_{i}H_{i}\}}^{2}\Big|\\
&\leq\Big(\vnorm{T_{\rho}Q_{f}^{\iota}\{G_{i}H_{i}\}}_{2,\mathcal{V}}+\vnorm{\mathrm{Chop}T_{\rho}Q_{f}^{\iota}\{G_{i}H_{i}\}}_{2,\mathcal{V}}\Big)\cdot
\vnorm{T_{\rho}Q_{f}^{\iota}\{G_{i}H_{i}\}-\mathrm{Chop}T_{\rho}Q_{f}^{\iota}\{G_{i}H_{i}\}}_{2,\mathcal{V}}\\
&\leq 20n^{1/2}\tau^{\frac{1-\rho}{30(c_{2}c_{3})}}
+O_{m,n}(\tau^{-1/3}p^{-1/2}).
\end{flalign*}
Therefore,
\begin{flalign*}
&\frac{1}{n}\Ebi\Tr\abs{T_{\rho}Q_{f}\{b_{i}\}}^{2}
\stackrel{\eqref{six30}\wedge\eqref{five3}}{=}
\frac{1}{n}\Eghi\Tr\abs{T_{\rho}Q_{f}^{\iota}\{G_{i}H_{i}\}}^{2}\\
&\qquad\qquad\qquad\qquad
\leq\frac{1}{n}\Eghi\Tr\abs{\mathrm{Chop}T_{\rho}Q_{f}^{\iota}\{G_{i}H_{i}\}}^{2}
+20n^{1/2}\tau^{\frac{1-\rho}{30(c_{2}c_{3})}}
+O_{m,n}(\tau^{-1/3}p^{-1/2}),
\end{flalign*}
proving \eqref{twelve0}.

Using $\Ebi Q_{f}^{\iota}\{b_{i}\}=\Eghi T_{\rho}Q_{f}^{\iota}\{G_{i}H_{i}\}=0$ and the Cauchy-Schwarz inequality,
\begin{flalign*}
\frac{1}{n}\Big|\Eghi\Tr\,\mathrm{Chop}T_{\rho}Q_{f}^{\iota}\{G_{i}H_{i}\}\Big|
&=\frac{1}{n}\Big|\Eghi\Tr\,\mathrm{Chop}T_{\rho}Q_{f}^{\iota}\{G_{i}H_{i}\}-\Eghi \Tr\,T_{\rho}Q_{f}^{\iota}\{G_{i}H_{i}\}\Big|\\
&\leq20n^{1/2}\tau^{\frac{1-\rho}{30(c_{2}c_{3})}}
+O_{m,n}(\tau^{-1/3}p^{-1/2}),
\end{flalign*}
using Corollary \ref{cor:smooth-invariance} again.
\end{proof}

\section{An Anti-Concentration Inequality}\label{secanti}

As in \cite{mossel10}, we can use our invariance principle to prove anti-concentration estimates of polynomials.

\begin{cor}[\embolden{An Anti-Concentration Estimate}]\label{lastcor}
There exists a constant $C>0$ such that the following holds.  Let $Q\colon (M_{n}(\C))^{m}\to M_{n}(\C)$ be a noncommutative multilinear polynomial of degree $d$.  Assume $\Ebi \frac{1}{n}\Tr\abs{Q\{b_{i}\}}^{2}\leq1$.  Let $\tau=\max_{1\leq j\leq m}\mathrm{Inf}_{j}(Q)$.  Define $\mathrm{Var}(Q)=\E_{b_{i}\sim\mathcal{B}}\Tr\absf{Q\{b_{i}\}-(\E_{b_{j}\sim\mathcal{B}}Q\{b_{j}\})}^{2}$.  Then, for any $t\in\R$,
\begin{flalign*}
\frac{1}{n}\mathop{\P}_{\substack{G_{i}\sim\mathcal{G}\otimes\Id\\ H_{i}\sim\mathcal{H}}}(\vnorm{Q^{\iota}\{G_{i}H_{i}\}}>t)
&\leq\mathop{\P}_{b_{i}\sim\mathcal{B}}(\vnorm{Q\{b_{i}\}}>t)
+O(n^{3}c_{3}^{d}\tau^{1/100})\\
&\qquad+Cd(4\tau^{1/100}n/[\mathrm{Var}(Q)]^{1/2})^{1/d}
+O_{m,n}(\tau^{-1/100}p^{-1/2}).
\end{flalign*}
\end{cor}
\begin{remark}
The $\frac{1}{n}$ term on the left side of the inequality seems to be an artifact of our proof method.  It comes from \eqref{thi2} below, where we bound the normalized trace of a matrix by its operator norm.
\end{remark}
\begin{proof}
Define $\phi\colon\R\to\R$ by $\phi(x)= c\cdot \exp\left(\frac{-1}{1-x^{2}}\right)$ if $\abs{x}<1$ and $\phi(x)=0$ for all other $x\in\R$.  The constant $1/2<c<4$ is chosen so that $\int_{\R}\phi(x)dx=1$.  It is well-known that $\phi$ is an infinitely differentiable function with bounded derivatives.

% note  2\geq \int_{-1}^{1}e^{1/(x^2 -1)}\geq \int_{-1/2}^{1/2} e^(1/(x^2 -1)) dx\geq e^(1/(-3/4)) = e^(-4/3)
% \abs{\phi'(x)}\leq (2/e)c \leq 3
% \abs{\phi''(x)}\leq 4*(4/e)^4 c\leq 76
% \abs{\phi'''(x)}\leq 4*22*(6/e)^6 \leq 10^5

Fix $r,s>0$.  Define $\psi\colon\R\to\R$ by
\begin{equation}\label{thi0}
\psi(x)=\begin{cases} 0  &\mbox{if}\,\,x\leq r-s\\ \frac{s-r+x}{2s}&\mbox{if}\,\,r-s\leq x\leq r+s\\ 1&\mbox{if}\,\, x>r+s.\end{cases}
\end{equation}
Define $\psi_{\lambda}(x)=\psi*\phi_{\lambda}(x)=\int_{\R}\psi(y)\phi_{\lambda}(x-y)dy$.  Then $\psi_{\lambda}(x)=\psi(x)$ for any $x\in\R$ with  $x>r+s+\lambda$ or $x<r-s-\lambda$.  So,
\begin{equation}\label{thi1}
\Eghi \Tr\psi_{\lambda}\abs{Q^{\iota}\{G_{i}H_{i}\}}^{2}
\geq \mathop{\P}_{\substack{G_{i}\sim\mathcal{G}\otimes\Id\\ H_{i}\sim\mathcal{H}}}(\vnorm{Q^{\iota}\{G_{i}H_{i}\}}>r+s+\lambda).
\end{equation}
\begin{equation}\label{thi2}
\Ebi \frac{1}{n}\Tr\psi_{\lambda}\abs{Q\{b_{i}\}}^{2}
\leq \mathop{\P}_{b_{i}\sim\mathcal{B}}(\vnorm{Q\{b_{i}\}}>r-s-\lambda).
\end{equation}
Note that $\abs{\psi_{\lambda}'''(x)}\leq10^{10}\lambda^{-2}s$.  Applying Theorem \ref{thm:invariance-convex},
\begin{equation}\label{thi3}
\begin{aligned}
&\Eghi \frac{1}{n}\Tr\psi_{\lambda}\abs{Q^{\iota}\{G_{i}H_{i}\}}^{2}\\
&\qquad\leq\Ebi \frac{1}{n}\Tr\psi_{\lambda}\abs{Q\{b_{i}\}}^{2}+s\lambda^{-2}n^{3/2}10^{10}\hyconstb\tau^{1/4}+O_{m,n}(\lambda^{-2}sp^{-1/2}).
\end{aligned}
\end{equation}
Combining \eqref{thi1}, \eqref{thi2} and \eqref{thi3},
\begin{flalign}
&\frac{1}{n}\mathop{\P}_{\substack{G_{i}\sim\mathcal{G}\otimes\Id\\ H_{i}\sim\mathcal{H}}}(\vnorm{Q^{\iota}\{G_{i}H_{i}\}}>r+s+\lambda)\nonumber\\
&\leq \mathop{\P}_{b_{i}\sim\mathcal{B}}(\vnorm{Q^{\iota}\{b_{i}\}}>r-s-\lambda)
+s\lambda^{-2}10^{10}\hyconstb\tau^{1/4}+O_{m,n}(s\lambda^{-2}p^{-1/2})\nonumber\\
&=\mathop{\P}_{b_{i}\sim\mathcal{B}}(\vnorm{Q\{b_{i}\}}>r+s+\lambda)
+\mathop{\P}_{b_{i}\sim\mathcal{B}}(r+s+\lambda>\vnorm{Q\{b_{i}\}}>r-s-\lambda)\nonumber\\
&\qquad\qquad\qquad+s\lambda^{-2}n^{3/2}10^{10}\hyconstb\tau^{1/4}+O_{m,n}(s\lambda^{-2}p^{-1/2})\label{thi6}.
\end{flalign}

It remains to show that the second term in \eqref{thi6} is small.  To this end we apply the anti-concentration result of \cite[Theorem 8]{carbery01} (with $q=2d$ in their notation) to get: there exists an absolute constant $C'>0$ such that, if $g_{1},\ldots,g_{m}$ are i.i.d. standard real Gaussian random variables, and if $Q$ is any noncommutative multilinear polynomial, then for all $\epsilon>0$,
$$\P_{g_{1},\ldots,g_{m}}(\vnorm{Q\{g_{i}\}}<\epsilon)\leq C' d(\epsilon/[\E_{g_{1},\ldots,g_{m}}\vnorm{Q\{g_{i}\}}^{2}]^{1/2})^{1/d}.$$
Since $\E_{g_{1},\ldots,g_{m}}\vnorm{Q\{g_{i}\}}^{2}\geq \E_{g_{1},\ldots,g_{m}}\frac{1}{n}\Tr\abs{Q\{g_{i}\}}^{2}\geq \E_{g_{1},\ldots,g_{m}}\frac{1}{n}\Tr\absf{Q\{g_{i}\}-\widehat{Q}(\emptyset)}^{2}$, we conclude that, for any $r\in\R$, we have the following ``small ball'' probability estimate.
\begin{equation}\label{thi4}
\P_{g_{1},\ldots,g_{m}}(\abs{\vnorm{Q\{g_{i}\}}-r}<\epsilon)\leq C' d(\epsilon n/[\E_{g_{1},\ldots,g_{m}}\Tr\absf{Q\{g_{i}\}-\widehat{Q}(\emptyset)}^{2}]^{1/2})^{1/d}.
\end{equation}
Now, applying the invariance principle \cite[Theorem 3.6]{isaksson11} with the function $\Psi\colon M_{n}(\C)\to\R$ defined by
$$\Psi(A)=\begin{cases} 0  &\mbox{if}\,\,\vnorm{A}\leq r-2s\\
\frac{2s-r+\vnorm{A}}{s}&\mbox{if}\,\,r-2s\leq \vnorm{A}\leq r-s\\
1&\mbox{if}\,\,r-s\leq \vnorm{A}\leq r+s\\
\frac{2s+r-\vnorm{A}}{s}&\mbox{if}\,\,r+s\leq \vnorm{A}\leq r+2s\\
0&\mbox{if}\,\, \vnorm{A}>r+2s.\end{cases}$$
We get
\begin{equation}\label{thi5}
\abs{\E_{b_{i}\sim\mathcal{B}}\Psi(Q\{b_{i}\})-\E_{g_{1},\ldots,g_{m}}\Psi(Q\{g_{i}\})}\leq \frac{2}{s}n^{3}C''\tau^{1/50}.
\end{equation}
So, applying the definition of $\Psi$ to \eqref{thi5}, we get
\begin{flalign*}
\P_{b_{i}\sim\mathcal{B}}(\abs{\vnorm{Q\{b_{i}\}}-r}<s)
&\leq \frac{2}{s}C''n^{3}\tau^{1/50}+\P_{g_{1},\ldots,g_{m}}(\abs{\vnorm{Q\{g_{i}\}}-r}<2s)\\
&\stackrel{\eqref{thi4}}{\leq}\frac{2}{s}C''n^{3}\tau^{1/50}+C' d(2s n/[\E_{g_{1},\ldots,g_{m}}\Tr\absf{Q\{g_{i}\}-\widehat{Q}(\emptyset)}^{2}]^{1/2})^{1/d}\\
&= \frac{2}{s}C''n^{3}\tau^{1/50}+C' d(2sn/[\E_{b_{i}\sim\mathcal{B}}\Tr\absf{Q\{b_{i}\}-\widehat{Q}(\emptyset)}^{2}]^{1/2})^{1/d}.
\end{flalign*}

Finally, substitute the last inequality into \eqref{thi6} and set $s=\lambda=\tau^{1/100}$.
\end{proof}

\begin{remark}
The theorem \cite[Theorem 3.6]{isaksson11} used in \eqref{thi5} provides an extra multiplicative factor of $2^{n^{2}}$ in \eqref{thi5}.  However, this constant can be removed in the following way.  Using their notation, they define a function $\phi\colon\R^{k}\to\R$ so that $\phi(x)=\exp(\frac{-1}{1-\vnorm{x}_{2}^{2}})$ if $\vnorm{x}_{2}<1$ and $\phi(x)=0$ otherwise.  (In the present paper, we set $k=n^{2}$.)  This is the function they use in their convolution formula.  If we instead use a function $\phi$ which is a product of one-dimensional functions, each of which is supported in the interval $[-1,1]$, e.g. $\phi(x)=\prod_{i=1}^{k}e^{\frac{-1}{1-x_{i}^{2}}}$, then the factor $2^{k}$ no longer appears in their proof.
\end{remark}

\begin{remark}
The stronger, though more restrictive anti-concentration inequality
\begin{equation}\label{lasteq}
\begin{aligned}
&\sup_{t\in\R}\Big|\mathop{\P}_{g_{1},\ldots,g_{m}}(\vnorm{Q\{g_{i}\}}>t)
-\mathop{\P}_{b_{i}\sim\mathcal{B}}(\vnorm{Q\{b_{i}\}}>t)\Big|\\
&\qquad\qquad\qquad\qquad\leq O(n^{3}c_{3}^{d}\tau^{1/100})
+Cd(2\tau^{1/100}n/[\mathrm{Var}(Q)]^{1/2})^{1/d},
\end{aligned}
\end{equation}
follows more directly from \cite[Theorem 8]{carbery01} and \cite[Theorem 3.6]{isaksson11} by repeating the argument above.  For example, if we use $\psi$ defined in \eqref{thi0}, then \cite[Theorem 3.6]{isaksson11} implies that
$$
\abs{\E_{b_{i}\sim\mathcal{B}}\psi(Q\{b_{i}\})-\E_{g_{1},\ldots,g_{m}}\psi(Q\{g_{i}\})}\leq \frac{2}{s}n^{3}C''\tau^{1/50}.
$$
Applying the definition of $\psi$ to this inequality, we get
\begin{equation}\label{thi7}
\Big|\mathop{\P}_{g_{1},\ldots,g_{m}}(\vnorm{Q\{g_{i}\}}>r+s+\lambda)
-\mathop{\P}_{b_{i}\sim\mathcal{B}}(\vnorm{Q\{b_{i}\}}>r-s-\lambda)\Big|\leq \frac{2}{s}n^{3}C''\tau^{1/50}.
\end{equation}
Therefore,
\begin{flalign*}
&\Big|\mathop{\P}_{g_{1},\ldots,g_{m}}(\vnorm{Q\{g_{i}\}}>r+s+\lambda)
-\mathop{\P}_{b_{i}\sim\mathcal{B}}(\vnorm{Q\{b_{i}\}}>r+s+\lambda)\Big|\\
&=\Big|-\mathop{\P}_{g_{1},\ldots,g_{m}}(r+s+\lambda>\vnorm{Q\{g_{i}\}}>r-s-\lambda)\\
&\qquad+\mathop{\P}_{g_{1},\ldots,g_{m}}(\vnorm{Q\{g_{i}\}}>r+s+\lambda)
-\mathop{\P}_{b_{i}\sim\mathcal{B}}(\vnorm{Q\{b_{i}\}}>r-s-\lambda)\Big|\\
&\leq\Big|\mathop{\P}_{g_{1},\ldots,g_{m}}(r+s+\lambda>\vnorm{Q\{g_{i}\}}>r-s-\lambda)\Big|\\
&\qquad+\Big|\mathop{\P}_{g_{1},\ldots,g_{m}}(\vnorm{Q\{g_{i}\}}>r+s+\lambda)
-\mathop{\P}_{b_{i}\sim\mathcal{B}}(\vnorm{Q\{b_{i}\}}>r-s-\lambda)\Big|.
\end{flalign*}
The second term is bounded by \eqref{thi7} and the first term is bounded by \eqref{thi4}, setting $s=\lambda=\tau^{1/100}$.
\end{remark}

\begin{remark}
It would be desirable to upgrade Corollary \ref{lastcor} and \eqref{lasteq} to the stronger inequalities presented in \cite{meka15}.  We leave this research direction to future investigations.
\end{remark}

\noindent\textbf{Acknowledgements}.  Thanks to Todd Kemp, Elchanan Mossel, Assaf Naor, Krzysztof Oleszkiewicz, Dimitri Shlyakhtenko and Thomas Vidick for helpful discussions.  Thanks also to the anonymous reviewer for several helpful comments and for finding several typos and mistakes.

\bibliographystyle{amsalpha}
\bibliography{12162011}

\newcommand{\etalchar}[1]{$^{#1}$}
\def\polhk#1{\setbox0=\hbox{#1}{\ooalign{\hidewidth
  \lower1.5ex\hbox{`}\hidewidth\crcr\unhbox0}}} \def\cprime{$'$}
  \def\cprime{$'$}
\providecommand{\bysame}{\leavevmode\hbox to3em{\hrulefill}\thinspace}
\providecommand{\MR}{\relax\ifhmode\unskip\space\fi MR }
% \MRhref is called by the amsart/book/proc definition of \MR.
\providecommand{\MRhref}[2]{%
  \href{http://www.ams.org/mathscinet-getitem?mr=#1}{#2}
}
\providecommand{\href}[2]{#2}
\begin{thebibliography}{KKMO07}

\bibitem[Bha97]{bhatia97}
R.~Bhatia, \emph{Matrix analysis}, Graduate Texts in Mathematics, Springer New
  York, 1997.

\bibitem[Bon70]{bonami70}
Aline Bonami, \emph{\'{E}tude des coefficients de {F}ourier des fonctions de
  {$L^{p}(G)$}}, Ann. Inst. Fourier (Grenoble) \textbf{20} (1970), no.~fasc. 2,
  335--402 (1971). \MR{0283496 (44 \#727)}

\bibitem[BR15]{rag15}
Jonah Brown{-}Cohen and Prasad Raghavendra, \emph{Combinatorial optimization
  algorithms via polymorphisms}, Electronic Colloquium on Computational
  Complexity {(ECCC)} \textbf{22} (2015), 7.

\bibitem[Cha06]{chat06}
Sourav Chatterjee, \emph{A generalization of the {L}indeberg principle}, Ann.
  Probab. \textbf{34} (2006), no.~6, 2061--2076. \MR{2294976 (2008c:60028)}

\bibitem[CW01]{carbery01}
Anthony Carbery and James Wright, \emph{Distributional and {$L^q$} norm
  inequalities for polynomials over convex bodies in {$\mathbb{R}^n$}}, Math.
  Res. Lett. \textbf{8} (2001), no.~3, 233--248. \MR{1839474}

\bibitem[DN14]{deya14}
Aur{\'e}lien Deya and Ivan Nourdin, \emph{Invariance principles for homogeneous
  sums of free random variables}, Bernoulli \textbf{20} (2014), no.~2,
  586--603. \MR{3178510}

\bibitem[DR13]{dirksen13}
Sjoerd Dirksen and {\'E}ric Ricard, \emph{Some remarks on noncommutative
  {K}hintchine inequalities}, Bull. Lond. Math. Soc. \textbf{45} (2013), no.~3,
  618--624. \MR{3065031}

\bibitem[DS01]{davidson01}
Kenneth~R. Davidson and Stanislaw~J. Szarek, \emph{Local operator theory,
  random matrices and {B}anach spaces}, Handbook of the geometry of {B}anach
  spaces, {V}ol. {I}, North-Holland, Amsterdam, 2001, pp.~317--366.
  \MR{1863696}

\bibitem[Gro72]{gross72}
Leonard Gross, \emph{Existence and uniqueness of physical ground states}, J.
  Functional Analysis \textbf{10} (1972), 52--109. \MR{0339722 (49 \#4479)}

\bibitem[Gro75]{gross75}
\bysame, \emph{Logarithmic {S}obolev inequalities}, Amer. J. Math. \textbf{97}
  (1975), no.~4, 1061--1083. \MR{0420249 (54 \#8263)}

\bibitem[HT99]{haagerup99}
U.~Haagerup and S.~Thorbj{\o}rnsen, \emph{Random matrices and {$K$}-theory for
  exact {$C^\ast$}-algebras}, Doc. Math. \textbf{4} (1999), 341--450
  (electronic). \MR{1710376 (2000g:46092)}

\bibitem[IM12]{isaksson11}
Marcus Isaksson and Elchanan Mossel, \emph{Maximally stable {G}aussian
  partitions with discrete applications}, Israel J. Math. \textbf{189} (2012),
  347--396. \MR{2931402}

\bibitem[Kal02]{kalai02}
Gil Kalai, \emph{A {F}ourier-theoretic perspective on the {C}ondorcet paradox
  and {A}rrow's theorem}, Adv. in Appl. Math. \textbf{29} (2002), no.~3,
  412--426. \MR{1942631}

\bibitem[Kan14]{kane14}
Daniel~M. Kane, \emph{The correct exponent for the {G}otsman-{L}inial
  conjecture}, Comput. Complexity \textbf{23} (2014), no.~2, 151--175.
  \MR{3212596}

\bibitem[Kho02]{Khot02}
Subhash Khot, \emph{On the power of unique 2-prover 1-round games}, Proceedings
  of the Thirty-Fourth Annual ACM Symposium on Theory of Computing (New York),
  ACM, 2002, pp.~767--775 (electronic). \MR{MR2121525}

\bibitem[KKL88]{kahn88}
Jeff Kahn, Gil Kalai, and Nathan Linial, \emph{The influence of variables on
  boolean functions}, Proc. of 29th Annual IEEE Symposium on Foundations of
  Computer Science, 1988, pp.~68--80.

\bibitem[KKMO07]{khot07}
Subhash Khot, Guy Kindler, Elchanan Mossel, and Ryan O'Donnell, \emph{Optimal
  inapproximability results for {MAX}-{CUT} and other 2-variable {CSP}s?}, SIAM
  J. Comput. \textbf{37} (2007), no.~1, 319--357. \MR{2306295 (2008d:68035)}

\bibitem[KN09]{khot09}
Subhash Khot and Assaf Naor, \emph{Approximate kernel clustering}, Mathematika
  \textbf{55} (2009), no.~1-2, 129--165. \MR{2573605 (2011c:68166)}

\bibitem[KN13]{khot11}
\bysame, \emph{Sharp kernel clustering algorithms and their associated
  grothendieck inequalities}, Random Structures \& Algorithms \textbf{42}
  (2013), no.~3, 269--300.

\bibitem[KVV14]{kal14}
Dmitry~S. Kaliuzhnyi-Verbovetskyi and Victor Vinnikov, \emph{Foundations of
  free noncommutative function theory}, Mathematical Surveys and Monographs,
  vol. 199, American Mathematical Society, Providence, RI, 2014. \MR{3244229}

\bibitem[MJC{\etalchar{+}}14]{mackey14}
Lester Mackey, Michael~I. Jordan, Richard~Y. Chen, Brendan Farrell, and Joel~A.
  Tropp, \emph{Matrix concentration inequalities via the method of exchangeable
  pairs}, Ann. Probab. \textbf{42} (2014), no.~3, 906--945. \MR{3189061}

\bibitem[MN15]{mossel12}
Elchanan Mossel and Joe Neeman, \emph{Robust optimality of {G}aussian noise
  stability}, J. Eur. Math. Soc. (JEMS) \textbf{17} (2015), no.~2, 433--482.
  \MR{3317748}

\bibitem[MNV16]{meka15}
Raghu Meka, Oanh Nguyen, and Van Vu, \emph{Anti-concentration for polynomials
  of independent random variables}, Theory Comput. \textbf{12} (2016), Paper
  No. 11, 16. \MR{3542863}

\bibitem[MOO10]{mossel10}
Elchanan Mossel, Ryan O'Donnell, and Krzysztof Oleszkiewicz, \emph{Noise
  stability of functions with low influences: invariance and optimality}, Ann.
  of Math. (2) \textbf{171} (2010), no.~1, 295--341. \MR{2630040 (2012a:60091)}

\bibitem[Mos10]{mossel10b}
Elchanan Mossel, \emph{Gaussian bounds for noise correlation of functions},
  Geom. Funct. Anal. \textbf{19} (2010), no.~6, 1713--1756. \MR{2594620
  (2011b:60080)}

\bibitem[MP14]{mendelson14}
Shahar Mendelson and Grigoris Paouris, \emph{On the singular values of random
  matrices}, J. Eur. Math. Soc. (JEMS) \textbf{16} (2014), no.~4, 823--834.
  \MR{3191978}

\bibitem[MS12]{meckes12}
Mark~W. Meckes and Stanis{\l}aw~J. Szarek, \emph{Concentration for
  noncommutative polynomials in random matrices}, Proc. Amer. Math. Soc.
  \textbf{140} (2012), no.~5, 1803--1813. \MR{2869165 (2012j:60048)}

\bibitem[Nel73]{nelson73}
Edward Nelson, \emph{The free {M}arkoff field}, J. Functional Analysis
  \textbf{12} (1973), 211--227. \MR{0343816 (49 \#8556)}

\bibitem[NPR10]{nourdin10}
Ivan Nourdin, Giovanni Peccati, and Gesine Reinert, \emph{Invariance principles
  for homogeneous sums: universality of {G}aussian {W}iener chaos}, Ann.
  Probab. \textbf{38} (2010), no.~5, 1947--1985. \MR{2722791 (2011g:60043)}

\bibitem[O'D14a]{odonnell14}
Ryan O'Donnell, \emph{Analysis of {B}oolean functions}, Cambridge University
  Press, 2014.

\bibitem[O'D14b]{odonnell14b}
\bysame, \emph{Social choice, computational complexity, gaussian geometry, and
  boolean functions}, Proceedings of the ICM, 2014.

\bibitem[Rot79]{rotar79}
V.~I. Rotar{\cprime}, \emph{Limit theorems for polylinear forms}, J.
  Multivariate Anal. \textbf{9} (1979), no.~4, 511--530. \MR{556909
  (81m:60039)}

\bibitem[Tal94]{talagrand94}
Michel Talagrand, \emph{On {R}usso's approximate zero-one law}, Ann. Probab.
  \textbf{22} (1994), no.~3, 1576--1587. \MR{1303654 (96g:28009)}

\bibitem[TV11]{tao11}
Terence Tao and Van Vu, \emph{Random matrices: universality of local eigenvalue
  statistics}, Acta Math. \textbf{206} (2011), no.~1, 127--204. \MR{2784665
  (2012d:60016)}

\bibitem[Ver12]{ver12}
Roman Vershynin, \emph{Introduction to the non-asymptotic analysis of random
  matrices}, Compressed sensing, Cambridge Univ. Press, Cambridge, 2012,
  pp.~210--268. \MR{2963170}

\bibitem[Voi91]{voiculescu91}
Dan Voiculescu, \emph{Limit laws for random matrices and free products},
  Invent. Math. \textbf{104} (1991), no.~1, 201--220. \MR{1094052 (92d:46163)}

\end{thebibliography}

\end{document}